\documentclass[10pt]{article}
\usepackage{latexsym}
\usepackage{amsmath}
\usepackage{comment}
\usepackage{soul}
\usepackage{multicol}
\usepackage{dsfont}
\usepackage{amssymb}
\usepackage{amsfonts}
\usepackage{amsthm}
\usepackage[margin=0.75in]{geometry}

\newtheorem{theorem}{Theorem}

\newtheorem{corollary}[theorem]{Corollary}
\newtheorem{lemma}[theorem]{Lemma}
\theoremstyle{definition}

\theoremstyle{remark}
\newtheorem{remark}[theorem]{Remark}

\usepackage{enumitem}
\usepackage{booktabs}
\usepackage{float}
\usepackage{xcolor}
\usepackage{algorithm}
\usepackage{algpseudocode}
\usepackage{caption}
\usepackage{cite}
\usepackage{ifthen}

\usepackage{tikz}
\usetikzlibrary{arrows.meta, bending, positioning, shapes, calc}

\usepackage{subcaption}
\usepackage{pgfplots}
\usepackage{pgfmath}
\usepackage{pgfplotstable}

\newcommand{\spacev}{{s}}
\newcommand{\Ob}{\Omega}

\usepackage{parskip}
\pgfplotsset{compat=1.18}
\usepgfplotslibrary{patchplots}

\graphicspath{ {./images/} }

\definecolor{myred}{RGB}{214, 39, 40}
\definecolor{myblue}{RGB}{31, 119, 180}
\definecolor{mygreen}{RGB}{44, 160, 44}
\definecolor{myorange}{RGB}{255, 127, 14}

\newcommand{\pd}[2]{\dfrac{\partial #1}{\partial #2}}

\newboolean{showimages}
\setboolean{showimages}{false}

\title{Online Gate-Driven Flow Control in Resin Transfer Moulding Using a Neural-Network Surrogate}
\author{Nicholas Wright$^{1,*}$, Oliver Maclaren$^{1}$, Piaras Kelly$^{1}$,\\Suresh Advani$^{2}$, Ruanui Nicholson$^{1}$\\[1.5ex]
{\small $^{1}$Department of Engineering Science, University of Auckland, Auckland, New Zealand}\\
{\small $^{2}$Center for Composite Materials, University of Delaware, Newark, DE, USA}\\[0.5ex]
{\small $^{*}$Corresponding author: \texttt{nwri968@aucklanduni.ac.nz}}}
\date{August 2026}

\begin{document}

\maketitle

\begin{abstract}
In resin transfer moulding, complete saturation of the fibre preform is necessary before the resin front reaches the outlet vent(s), to prevent dry-spot formation. In practice, the flow front rarely advances uniformly due to race-tracking effects. We propose a combined estimation and control strategy to address this issue. We use pressure-sensor data collected during filling to estimate the unknown race-tracking strengths via an iterated extended Kalman filter, and to simultaneously optimise auxiliary gate pressures to prevent the resin from arriving at the vent before complete saturation occurs. To make this approach feasible in real time, we replace the expensive finite element-control volume model with neural network surrogate models. The Bayesian approximation error framework is used to account for the discrepancy between the surrogates and the original model. To make the training of the surrogates feasible with time-dependent control actions, we prove that the flow-front geometry depends only on the time-averaged gate pressures. We demonstrate the methodology on a representative forked geometry, showing that the combined estimation-and-control framework can substantially reduce the number of unfilled nodes at the end of filling compared with an uncontrolled baseline. However, the level of improvement depends strongly on the location and number of auxiliary gates.
\end{abstract}

\section{Introduction}

Resin transfer moulding (RTM) is a widely used manufacturing technique for producing near-net-shaped fibre-reinforced composite components. The process involves injecting a low-viscosity thermoset resin through one or more gates into a closed mould containing a dry fibre preform, allowing the resin to impregnate the reinforcement before curing. A schematic of the process is shown in Figure~\ref{fig:RTMschematic}.

\begin{figure}[t]
    \centering
    \includegraphics[width=\textwidth]{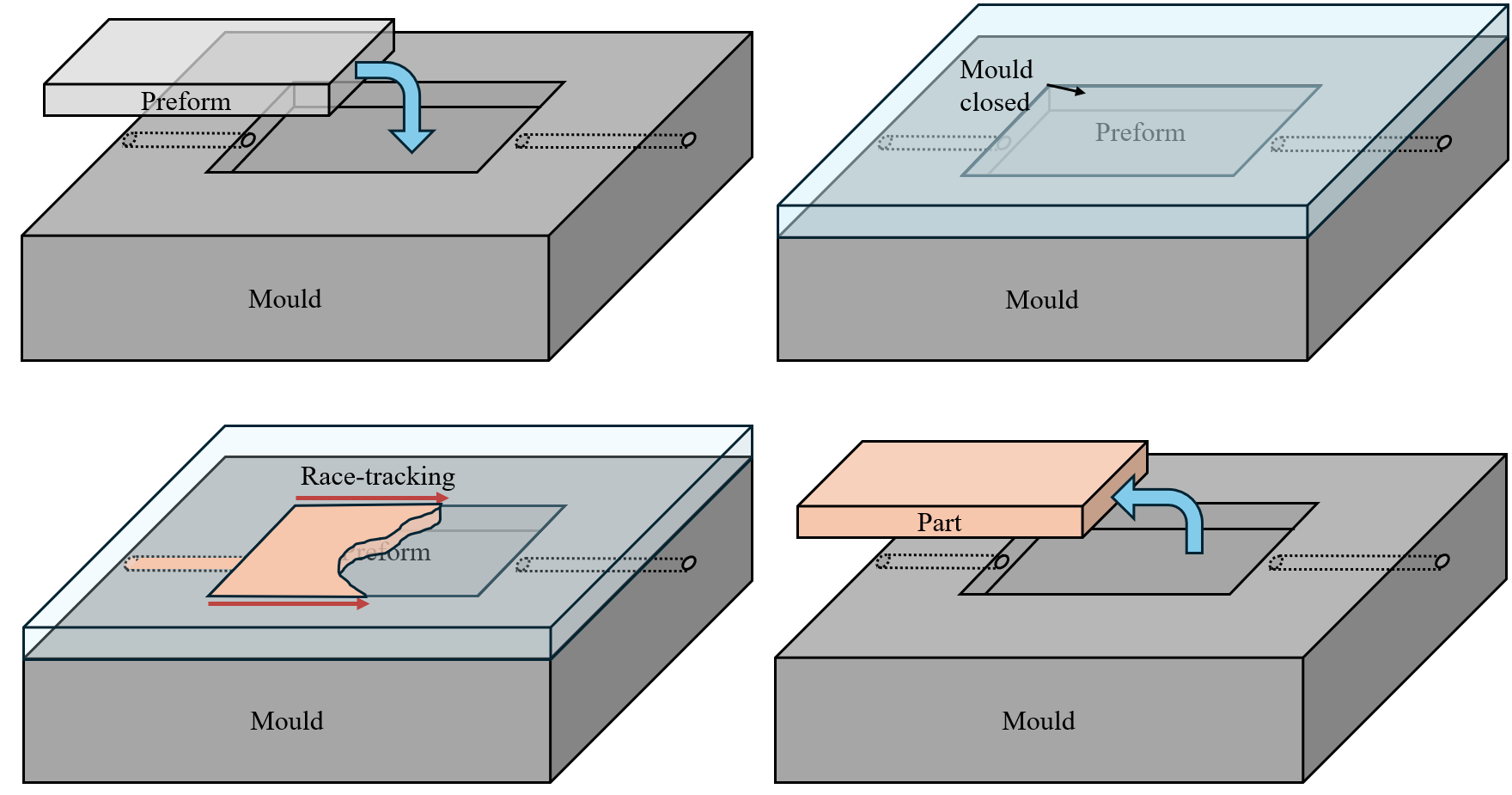}
    \caption{Schematic of the resin transfer moulding process. Resin is injected under pressure through one or more gates into a closed mould containing a dry fibre preform and flows toward the outlet vent(s). Race-tracking channels along preform edges can cause the flow front to advance non-uniformly, potentially reaching the vent before the preform is fully saturated. Note that in general the top of the mould is not transparent, but is here for illustrative purposes.}
    \label{fig:RTMschematic}
\end{figure}

A critical requirement during filling is that the resin must completely saturate the preform before any portion of the flow front reaches all outlet vents. If resin arrives at all vent(s) prematurely it seals any remaining air inside the mould, producing permanent dry spots that become voids upon curing and significantly degrade the structural integrity and mechanical performance of the final part~\cite{deAlmeida1994, Wei2019, Mehdikhani2019}. In practice the flow front rarely advances uniformly: the principal source of distortion is \textit{race tracking}, whereby regions of elevated permeability near preform edges, curved mould sections, or reinforcement boundaries form preferential flow channels~\cite{Bickerton1999, RTCharacterisation, bickerton2000, Kumar2014}. Because race-tracking strengths depend on lay-up, preform placement, fabric cutting, and mould geometry, they vary stochastically between nominally identical parts~\cite{RTCharacterisation, lawrence2004}, so a fixed injection strategy that fills one part may leave dry spots in another. Moreover, in general the RTM mould is opaque and such distortions cannot be easily detected. These features motivate \textit{online} estimation and control strategies that identify race-tracking conditions during filling and adjust the injection parameters in real time to direct the flow toward complete saturation.

Computational modelling plays a central role in understanding and mitigating these issues. The predominant approach for simulating resin flow in RTM is the finite element-control volume (FECV) method~\cite{Bruschke, Kang, Lee}, in which the pressure field is solved using finite elements on a fixed mesh while the advancing flow front is tracked through an overlapping set of control volumes. Several widely used solvers implement this approach, including the Liquid Injection Moulding Simulator (LIMS)~\cite{LIMS, LIMS2}. FECV-based simulations have been instrumental in the design of gate and vent locations~\cite{Ye2004, Ali2002}, the study of race-tracking effects~\cite{bickerton2000, Bickerton1999}, and the development of control strategies~\cite{Sozer2000, Part1, Part2}. However, these simulations are computationally expensive, as the pressure problem must be re-solved each time a control volume fills, and the event-driven time stepping introduces discontinuities in the model output with respect to parameters such as permeability~\cite{Wright2025}. These characteristics make direct use of FECV models challenging for real-time inference and control, motivating the use of surrogate models.

Active control of RTM filling has been studied for at least three decades. Early numerical work by Liu et al.~\cite{Liu} showed that opening and closing gates in response to the evolving flow front can prevent dry-spot formation. Subsequent studies developed on-line controllers that map sensor measurements to corrective gate actions and validated them experimentally, including strategic decision-tree control~\cite{Sozer2000}, real-time gate-flow-rate control under race tracking~\cite{Bickerton2001}, distributed sensing-and-actuation frameworks~\cite{Part1, Part2}, and disturbance rejection using combined sensors and actuators~\cite{Lawrence2003}. Such strategies are predetermined a priori with the aid of computational models, but the model cannot be used when the filling begins due to the expensive computational cost. This limits the flexibility and capability of such methods.

A parallel body of work places a computational model within the control loop. Nielsen and Pitchumani~\cite{Nielsen2002} used on-the-fly finite-difference simulations to set injection-port flow rates in real time; Kang et al.~\cite{Kang} pre-computed multi-gate injection schedules from finite element simulations; Restrepo et al.~\cite{Restrepo2007} developed adaptive flow-rate and pressure-control algorithms; and Neitzel and Puch~\cite{Neitzel2023} combined capacitive sensing with controlled pressure sequences to minimise void formation. These approaches improve fill quality but generally rely on full flow-front visualisation or pre-computed strategies that do not adapt to unknown, variable permeability fields.

Machine-learning methods have more recently been applied to RTM defect detection~\cite{Mendikute2021, Wang2024}, race-tracking identification from sensor data~\cite{Leon2022, Ago, Siddig2018}, process monitoring~\cite{Torres2019, Crawford2020}, and porosity inference within hybrid-twin frameworks~\cite{Rodriguez2023}. Closest to the present work, reinforcement learning (RL) has been explored for RTM flow control: Szarski and Chauhan~\cite{Szarski2023} optimised flow-distribution networks using deep RL, and Mocerino et al.~\cite{Mocerino2025} demonstrated the first sim-to-real transfer of an RL-based pressure controller, using RL, trained in simulation and deployed on a laboratory-scale setup. These policies are learned offline to be robust across a distribution of material variability; they do not estimate the uncertain permeability online and therefore provide no uncertainty quantification over the race-tracking parameters, which is valuable for quality assessment and process understanding. The present work instead couples estimation and control, providing calibrated permeability uncertainty throughout filling while optimising the gate pressures to prevent dry spots.

Other research similar to what we present here pursues online Bayesian estimation of the permeability parameters. Iglesias et al.~\cite{Iglesias1, Iglesias2, Iglesias3} applied Bayesian inversion to spatially variable permeability fields in RTM, most recently using neural network surrogates for real-time inversion~\cite{Iglesias3}. In our earlier work we estimated race-tracking strengths from in-mould pressure measurements and introduced an optimal sensor-placement strategy~\cite{Wright2023}; showed that a neural network surrogate overcomes the discontinuity and non-differentiability of the finite element-control volume method while Bayesian approximation error (BAE) accounts for the surrogate-solver discrepancy~\cite{Wright2025}; and extended this surrogate-and-BAE framework to generally varying permeability fields~\cite{Wright2026}.

In the present work, we bring these lines of research together: we combine online Bayesian inference of race-tracking parameters with active gate-pressure control to prevent dry-spot formation during mould filling. The core idea is to use pressure-sensor data collected during filling to estimate the unknown race-tracking strengths via an iterated extended Kalman filter (IEKF), and to simultaneously optimise auxiliary gate pressures so as to steer the flow front away from premature vent arrival. To make this approach feasible in real time, we replace the computationally expensive LIMS-based forward model with two neural network surrogates: one that predicts sensor pressures as a function of the permeability parameters and gate actions, and a second that predicts a continuous measure of the unfilled area used as the control objective. To make this possible, we exploit a unique property of quasi-static Darcy flow to resolve the curse of dimensionality introduced by time-varying control actions: we prove that the resin domain geometry depends exactly on the time-averaged gate pressures, allowing the injection history to be represented by a single scalar and making surrogate-based control tractable.

The remainder of the paper is organised as follows. Section~\ref{sec:ForwardModel} describes the mathematical model of RTM filling, including the governing equations, the race-tracking parameterisation, and the finite element-control volume solution method. Next, Section~\ref{sec:Control} formulates the Bayesian inverse problem and the online IEKF-based estimation, defines the gate-pressure control objective, and then develops the BAE-corrected surrogate, the average-pressure parameterisation that makes surrogate training tractable. It then finishes by describing the complete surrogate-integrated control algorithm in full. In Section~\ref{sec:Fork}, we present a numerical case study, including surrogate training details and numerical results. We conclude in Section~\ref{sec:Con}.

\section{Mathematical Modelling in RTM}
\label{sec:ForwardModel}

This section summarises the numerical model used to simulate the RTM filling process. The governing equations and numerical method are standard in liquid composite moulding; we therefore give only those details needed for the control formulation and refer the reader to~\cite{CompositesBook, Wright2023, Wright2025} for full accounts.

\subsection{Governing equations}

To begin, let $D \subset \mathbb{R}^2$ be a preform (saturated or unsaturated, within the mould cavity) with $n_g$ injection gates $\Gamma_{\text{in}}^{j}$ ($j=1,\dots,n_g$) and one or more outlet vents. In this work we consider a single vent, denoted $\Gamma_{\text{vent}}$. At any time $t \ge 0$ the preform can be partitioned into the resin-filled region $\Ob(t)$ and the remaining dry region $D \setminus \Ob(t)$, with the resin domain $\Ob(t) \subseteq D$ having free boundary $\Gamma_{\text{front}}(t) := \partial\Ob(t) \setminus \partial D$. For the slow, incompressible thermoset resins typically used in RTM, resin flow through the fibrous preform is well described by Darcy's law~\cite{CompositesBook, pillai1998}. The volumetric resin flux is given by
\begin{equation}
    q(\spacev, t) = -\frac{K(\spacev)}{\mu}\nabla p(\spacev, t), \quad \spacev \in \Ob(t),
    \label{eq:Darcy}
\end{equation}
where $p$ is the pressure, $\mu>0$ the resin viscosity, and $K \in \mathbb{R}^{2 \times 2}$ the symmetric, positive-definite permeability tensor. Conservation of mass under the assumption of quasi-static flow yields the elliptic pressure equation
\begin{equation}
    -\nabla \cdot \left(\frac{K(\spacev)}{\mu}\nabla p(\spacev, t)\right) = 0, \quad \spacev \in \Ob(t).
    \label{eq:Pressure}
\end{equation}
$\Omega(t)$ is subject to a zero-flow Neumann boundary condition on the mould walls, a $p=0$ Dirichlet boundary at the flow front and Dirichlet boundaries at injection gates: the time-varying gate pressures $a_j(t)$ on $\Gamma_{\text{in}}^{j}$. Specifically, we have
\begin{align}
    p(\cdot, t) &= a_j(t) && \hspace{-3cm}\text{on } \Gamma_{\text{in}}^{j}, \quad j = 1, \dots, n_g, \label{eq:BVP_gates} \\
    p(\cdot, t) &= 0 && \hspace{-3cm}\text{on } \Gamma_{\text{front}}(t) \cup \Gamma_{\text{vent}}, \label{eq:BVP_front} \\
    q(\cdot,t) \cdot n &= 0 && \hspace{-3cm}\text{on } \partial D \setminus ( \Gamma_{\text{vent}} \cup (\cup_{j=1}^{n_g} \Gamma_{\text{in}}^{j})), \label{eq:BVP_walls}
\end{align}
where $n$ is the normal vector to the impermeable mould wall.

The advancement of the resin domain is governed by the kinematic boundary condition at the flow front, known in free-boundary theory as the Stefan condition~\cite{Elliott1982}. The normal front velocity $v_n$ is determined by
\begin{equation}
    v_n(\spacev, t) = \frac{q(\spacev, t) \cdot n(\spacev, t)}{\phi} = -\frac{K(\spacev)}{\mu\phi} \frac{\partial p}{\partial n}(\spacev, t), \quad \spacev \in \Gamma_{\text{front}}(t),
    \label{eq:FrontVelocity}
\end{equation}
where $\phi$ is the porosity of the preform and $\pd{p}{n}$ is the directional derivative of the pressure field in the normal direction to the flow boundary. Note that the pressure $p$, flux $q$, and front normal $n$ all depend on time through the evolving resin domain $\Ob(t)$.

The dominant source of flow-front distortion in RTM is race tracking (RT): the formation of preferential flow channels along preform edges, mould boundaries, or other regions of elevated permeability~\cite{Bickerton1999, RTCharacterisation, bickerton2000}. We model RT by partitioning the mould geometry into $n$ potential race-tracking regions, each with an associated \emph{race-tracking strength} that scales the local permeability relative to the nominal bulk value. We collect these strengths into a parameter vector $x \in \mathbb{R}^n$ and work with the log-permeabilities, so that $x_i = \log(K_i / K_{\rm bulk})$ for the $i$-th region. Here, $K_i$ is the permeability in race-tracking region $i$, while $K_{\rm bulk}$ is the permeability in the no-RT region of the preform.

\subsection{Finite element-control volume method}
\label{sec:FECV}

For a given resin domain $\Ob(t)$, the pressure field is computed using a finite element discretisation. The evolution of the resin domain $\Ob(t)$ is tracked using the finite element-control volume (FECV) method, widely adopted in liquid composite moulding simulations~\cite{Bruschke, Kang, Lee, CompositesBook}. In this approach, each node of the finite element mesh is associated with a control volume whose saturation $\psi_i \in [0,1]$ is tracked. The pressure solution provides resin fluxes between neighbouring control volumes via Darcy's law; these fluxes determine the rate at which each control volume fills. Time advances in an event-driven fashion: the next time step corresponds to the instant at which a control volume becomes fully saturated ($\psi_i = 1$), at which point $\Ob(t)$ is updated and the pressure problem is resolved on the expanded domain. In this work, simulations are carried out using the Liquid Injection Moulding Simulator (LIMS)~\cite{LIMS, LIMS2}, a widely used implementation of the FECV method. LIMS uses linear Lagrange finite elements for the pressure computation.

\subsection{Forward operator and observation model}
\label{sec:ObsModel}

Although the FECV model can solve for the pressure at all locations, in practice it is only possible to measure the pressure at discrete sensor locations within the mould. In the following section, it will be useful to define the injection-pressure history at all gates by time $t$ as
\begin{equation}
    A(t) \in \mathbb{R}^{n_g\times k} = a(\tau) \text{ for all } \tau \le t, \quad a(\tau) = [a_1(\tau), a_2(\tau), \dots]^T,
\end{equation}
where $k$ is the number of timesteps that have occurred at time $t$, i.e. instances where the injection pressures can be changed during online control. We then consider a \emph{forward operator} $f$ that maps the race-tracking strength $x$, time $t$ and the injection pressure history $A(t)$ to the simulated pressures at those sensor locations:
\begin{equation}
    f: \mathbb{R}^n\times\mathbb{R}_{\geq 0}\times\mathbb{R}^{n_g\times k} \to \mathbb{R}^m,
    \label{eq:ForwardOp}
\end{equation}
where $m$ is the number of pressure sensors. This forward operator represents a FECV simulation in LIMS using the race-tracking strengths $x$ and the time-dependent injection pressures given by $A(t)$ until time $t$. At this point, the solved finite element pressures are used to return the pressures at the specific sensor locations.

In practice, the measurements observed by the sensors contain some level of noise. We will assume that at a measurement time $t$, the observed sensor pressures $y_t \in \mathbb{R}^m$ are related to the race-tracking parameters through the additive error model
\begin{equation}
    y_t = f(x,t,A(t)) + e,
    \label{eq:ObsModel}
\end{equation}
where $e \in \mathbb{R}^m$ is measurement noise, that we assume to be independent of time, the injection pressures and the race-tracking strengths.

\section{Online Estimation and Control}
\label{sec:Control}

We now develop the online estimation and control framework. We first summarise the Bayesian approach to estimating the race-tracking parameters $x$ from in-mould pressure measurements and the IEKF framework used to compute this estimation iteratively, online~\cite{Kaipio, calvetti2007introduction, stuart2010inverse}, extending our earlier formulation~\cite{Wright2023, Wright2025} to incorporate gate-pressure control, and the optimal selection of those gate pressure(s). We then describe the surrogate methodology needed to make this approach tractable in real time, including the reduction of the injection pressure history into only the average and current pressures. Finally, we combine these methods to introduce the complete algorithm.

\subsection{Bayesian approach for parameter estimation}
\label{sec:MAP}

The Bayesian approach to parameter estimation treats the unknown parameters (in this case, the race-tracking strengths) $x$ as random variables and seeks to infer their probability distribution given observed data $y_t$. Note that this inference is designed for a single time $t$ using one set of sensor pressures $y_t$. This inference is based on Bayes' formula:
\begin{equation}
    \pi(x|y_t) \propto \pi(y_t|x)\pi(x).
\end{equation}
Here, $\pi(y_t|x)$ is the likelihood distribution, and $\pi(x)$ is the prior distribution, which encodes our prior beliefs, i.e. before any data is received, in the race-tracking strengths. For the prior distribution, we restrict ourselves to a Gaussian distribution. Specifically, we set
\begin{equation}
    \pi(x) = \mathcal{N}(x_0, \Gamma_x),
    \label{eq:Prior}
\end{equation}
where $x_0$ and $\Gamma_x$ are the mean and covariance of the distribution, respectively. The form of the likelihood is based on the distribution of additive noise. We assume $e \sim \mathcal{N}(0, \Gamma_e)$, independent from the parameters (race-tracking strengths) and any control actions. In this case~\cite{Kaipio, kaipio2007statistical}, the likelihood takes the form
\begin{equation}
    \pi(y_t|x) \propto \exp\!\left\{-\tfrac{1}{2}\|L_e(y_t - f(x,t,A(t)))\|^2\right\},
    \label{eq:Likelihood}
\end{equation}
where $L_e^T L_e = \Gamma_e^{-1}$. Bayes' theorem then yields the \emph{posterior distribution}
\begin{equation}
    \pi(x|y_t) \propto \exp\!\left\{-\tfrac{1}{2}\|L_e(y_t - f(x,t,A(t)))\|^2 - \tfrac{1}{2}\|L_x(x - x_0)\|^2\right\},
    \label{eq:Posterior}
\end{equation}
where $L_x^T L_x = \Gamma_x^{-1}$.

For the nonlinear, real-time problem considered here, accurately characterising the posterior - for example via Markov chain Monte Carlo - is computationally infeasible. Instead, as is common, we take a Gaussian approximation centred at the maximum a posteriori (MAP) estimate, the value of $x$ that maximises the posterior, so that $\pi(x|y_t) \approx \mathcal{N}(x_{\rm MAP},\Gamma_{\rm post})$. The MAP estimate is equivalently the minimiser of the negative log posterior:
\begin{equation}
    x_{\rm MAP} = \arg\min_{x \in \mathbb{R}^n}\;\tfrac{1}{2}\|L_e(y_t - f(x,t,A(t)))\|^2 + \tfrac{1}{2}\|L_x(x - x_0)\|^2.
    \label{eq:MAP}
\end{equation}
On the other hand, linearising $f$ about $x_{\rm MAP}$ gives the approximate posterior covariance
\begin{equation}
    \Gamma_{\rm post} = \bigl(F^T \Gamma_e^{-1} F + \Gamma_x^{-1}\bigr)^{-1},
    \label{eq:PostCov}
\end{equation}
where $F := D_x f(x_{\rm MAP},t,A(t)) \in \mathbb{R}^{m \times n}$ denotes the Jacobian of the forward model, i.e., $[F]_{ij} = \partial f_i / \partial x_j$. For our forward model $f$ (the FECV numerical model), no closed-form or efficient (e.g. adjoint-based) method exists for computing the Jacobian $F$. It must instead be approximated using finite differences, which is computationally expensive and can be inaccurate owing to the non-smoothness of $f$, as discussed in detail in a previous work~\cite{Wright2025}.

\subsection{Sequential estimation}
\label{sec:Sequential}
The Bayesian formulation above is posed for an estimate of $x$ from one measurement vector $y_t$ at a single time $t$. During RTM filling, however, measurements arrive sequentially at times $t_1, t_2, \dots$, and we require online updates of $x$ together with online control decisions. We therefore perform sequential estimation using IEKF~\cite{Maybeck1979, Bergman1999, Wan2000, Simon2006}. At each time step $t=t_k$, the posterior from the previous time is used as the prior for the current time, and the update uses only the current measurement vector $y_k$ and the model $f$ evaluated at $t_k$; the posterior covariance therefore accumulates the information contained in all past measurements without requiring them to be stored.

Computing the MAP estimate~\eqref{eq:MAP} at each step requires solving a nonlinear least-squares problem. To solve this, we require an optimisation algorithm that can minimise the objective given in \eqref{eq:MAP}. This is a computationally critical part of the algorithm, as this optimisation will involve many computations of $f$ and the gradients of $f$. Here, we use Gauss--Newton~\cite{Nocedal2006, Aster2018}, as each Gauss--Newton step linearises the forward model and solves the same normal equations that define the MAP estimate and posterior covariance of a Gaussian, so the same linearisation that drives the optimisation also yields the posterior covariance~\eqref{eq:PostCov} on convergence, at no additional cost~\cite{Aster2018}. At step $k$, the MAP estimate $x_{\rm MAP}^{(k)}$ is computed with a small number of Gauss--Newton iterations initialised at the previous estimate $x_{\rm MAP}^{(k-1)}$, and the posterior covariance $\Gamma_x^{(k)}$ is then formed by linearising about $x_{\rm MAP}^{(k)}$. Because several Gauss--Newton iterations are performed per measurement update, this is an iterated EKF rather than a standard EKF, which performs a single linearised update step. For background on Kalman filtering and its extended variants, see~\cite{Kalman1960, Gelb1974, BarShalom2001, Maybeck1979, Simon2006}.

\subsection{Gate-pressure control}
\label{sec:ControlObjective}

The aim of control here is to minimise the resin left unsaturated at the end of filling. Recalling the control-volume saturations $\psi_i(t) \in [0,1]$ from Section~\ref{sec:FECV}, we define the terminal unfilled measure as the sum of the unsaturated fractions of all $N$ control volumes at the end of filling (i.e., at a time $T$):
\begin{equation}
    \Psi := \sum_{i=1}^{N} \bigl(1 - \psi_i(T)\bigr).
    \label{eq:ControlObjectiveDef}
\end{equation}
Because the mesh elements are approximately uniform in size, $\Psi$ is roughly proportional to the total unfilled area. A perfect fill corresponds to $\Psi = 0$.

Control action selection is based on the current MAP estimate $x_{\rm MAP}^{(k)}$ of the race-tracking strengths, rather than considering a range of feasible scenarios across the posterior distribution. Given $x_{\rm MAP}^{(k)}$, we consider a \emph{control model}
\begin{equation}
    \Psi(x, t, A(t), a_f) \in \mathbb{R}_{\geq 0},
    \label{eq:ControlModel}
\end{equation}
which takes the parameter estimate, the current fill time $t$, the injection pressure history $A(t)$ summarising the current fill state, and a candidate \emph{continuation gate pressure} $a_f$, and predicts the terminal unfilled measure~\eqref{eq:ControlObjectiveDef} that would result. The optimal control action is the continuation pressure(s) of the gate(s) that minimises the predicted terminal dry measure,
\begin{align}
    a_k^\star = \arg\min_{a_f \in \mathcal{A}} \Psi\!\bigl(x_{\rm MAP}^{(k)}, t_k, A(t), a_f\bigr),
    \label{eq:control_opt}
\end{align}
where $\mathcal{A}$ is the feasible set of gate pressures. We deliberately formulate~\eqref{eq:control_opt} as a pure minimisation of the predicted terminal dry measure, with no regularisation or penalty on the control effort (for example on the change $\Delta a$ between successive actions); the consequences of this are discussed in Section~\ref{sec:FullAlgorithm}. 

This leads to an algorithm of the form of receding-horizon control~\cite{Mayne2000, Rawlings2017}, where optimal continuous control action is computed, but is only applied for one time-step before the algorithm recomputes the next best injection pressure, until filling time is complete. The full sequential loop is summarised in Algorithm~\ref{alg:online_iekf_control}.

\begin{algorithm}[H]
\caption{Online IEKF estimation}
\label{alg:online_iekf_control}
\begin{algorithmic}[1]
\Require Prior mean and covariance $(x_{\rm MAP}^{(0)},\Gamma_x^{(0)})$, models $f,\Psi$, feasible set $\mathcal{A}$
\For{$k=1,2,\dots$ until filling completes}
    \State Acquire sensor data $y_k$ at time $t_k$
    \State Compute $x_{\rm MAP}^{(k)}$ with IEKF (Gauss--Newton iterations initialised at $x_{\rm MAP}^{(k-1)}$)
    \State Compute covariance $\Gamma_x^{(k)}$ from linearisation of the model at $x_{\rm MAP}^{(k)}$
    \State Compute control action $a_k^\star \gets \arg\min_{a\in\mathcal{A}} \Psi\!\bigl(x_{\rm MAP}^{(k)},t_k,a(\cdot)|_{\tau\leq t_k},a\bigr)$
    \State Apply $a_k^\star$ on $[t_k,t_{k+1})$
    \State Record applied action $a_k^\star$ to maintain injection history $a(\cdot)|_{\tau \leq t_{k+1}}$
\EndFor
\end{algorithmic}
\end{algorithm}

Evaluating the control objective $\Psi$ and the forward operator $f$ within this online loop using the FECV solver is computationally prohibitive. Each IEKF update requires repeated evaluations of $f$ and its Jacobian, and the control optimisation~\eqref{eq:control_opt} requires many evaluations of $\Psi$ over candidate gate pressures; every such evaluation is a full filling simulation. The FECV model is nondifferentiable, requiring costly finite-difference Jacobians. Moreover, the FECV behaves discontinuously with respect to the parameters~\cite{Wright2025}, meaning that the finite-difference Jacobians can prove to be inaccurate. Even for the simple geometry considered later these costs are incompatible with real-time operation, and for the complex geometries typical of industrial RTM they become infeasible. This motivates replacing both $f$ and $\Psi$ with fast surrogate models that evaluate in milliseconds and have readily available gradients, which we develop next.

\subsection{Optimal control with a surrogate}
\label{sec:BAE}

To make the online algorithm tractable we replace the forward operator $f$ and the control model $\Psi$ with fast, differentiable surrogates, denoted $g \approx f$ and $h \approx \Psi$ respectively; their architecture and training are described in Section~\ref{sec:Surrogate}. Both evaluate in milliseconds and supply exact Jacobians via automatic differentiation, so they may be used directly within the Gauss--Newton updates and the control optimisation~\eqref{eq:control_opt}. Replacing the original FECV model with a surrogate, however, introduces a systematic discrepancy between $g$ and $f$ that, if ignored, typically leads to erroneous parameter estimates and their associated uncertainty~\cite{Wright2026}. We account for it using the Bayesian approximation error (BAE) framework~\cite{Kaipio, Kolehmainen, nicholson2023global}, which rewrites the observation model in terms of $g$:
\begin{equation}
    y_t = f(x,t,A(t)) + e = g(x,t,A(t)) + \underbrace{\bigl(f(x,t,A(t)) - g(x,t,A(t))\bigr)}_{\varepsilon(x)} + e = g(x,t,A(t)) + \eta(x),
    \label{eq:BAEmodel}
\end{equation}
where $\varepsilon(x) := f(x) - g(x)$ is the \emph{approximation error} and $\eta(x) = \varepsilon(x) + e$ is the \emph{total error}. Note that these are both assumed to be independent of the measurement time $t$ and injection pressures. However, the distribution of $\varepsilon$ is conditional on $x$ and is approximated as Gaussian, $\varepsilon \mid x \sim \mathcal{N}(\varepsilon_{0|x},\, \Gamma_{\varepsilon|x})$, with
\begin{equation}
\begin{aligned}
    \varepsilon_{0|x} &= \varepsilon_0 + \Gamma_{\varepsilon x}\,\Gamma_x^{-1}(x - x_0), \\
    \Gamma_{\varepsilon|x} &= \Gamma_\varepsilon - \Gamma_{\varepsilon x}\,\Gamma_x^{-1}\,\Gamma_{x\varepsilon},
\end{aligned}
\label{eq:BAEstats}
\end{equation}
where the statistics $\varepsilon_0$, $\Gamma_\varepsilon$ and $\Gamma_{\varepsilon x}$ are estimated offline using Monte Carlo samples. Specifically, an ensemble $x^{(1)}, \dots, x^{(\ell)}$ is drawn from the prior~\eqref{eq:Prior} and each sample is evaluated under both $f$ (LIMS) and $g$ (the surrogate) to produce approximation-error samples
\begin{equation}
    \varepsilon^{(i)} = f(x^{(i)},t^{(i)},A^{(i)}(t^{(i)})) - g(x^{(i)},t^{(i)},A^{(i)}(t^{(i)}), \quad i = 1, \dots, \ell.
    \label{eq:BAEsamples}
\end{equation}
Note that we must also generate sampled inference times and control action (injection pressure) histories for each. These should be sampled to cover the range of histories and times that might be seen during online control. The methodology we use here for generating those samples is discussed in Section~\ref{sec:FullAlgorithm}.

The required means and covariances are then computed from these samples in the standard way~\cite{Kaipio}. Incorporating the approximation error yields the BAE-corrected posterior
\begin{equation}
    \pi_{\rm BAE}(x | y) \propto \exp\!\left\{-\tfrac{1}{2}\|L_{\eta|x}(y - g(x,t,A(t)) - \varepsilon_{0|x})\|^2 - \tfrac{1}{2}\|L_x(x - x_0)\|^2\right\},
    \label{eq:PosteriorBAE}
\end{equation}
where $L_{\eta|x}^T L_{\eta|x} = \Gamma_{\eta|x}^{-1} = (\Gamma_{\varepsilon|x} + \Gamma_e)^{-1}$. Because $g(x)$ is nonlinear, this posterior is not strictly Gaussian. However, consistent with Section~\ref{sec:MAP}, we adopt a Gaussian approximation centred at the maximum a posteriori (MAP) estimate. The BAE-corrected MAP estimate and the approximate posterior covariance are
\begin{equation}
    x_{\rm MAP}^{\rm BAE} = \arg\min_{x \in \mathbb{R}^n}\;\tfrac{1}{2}\|L_{\eta|x}(y - g(x,t,A(t)) - \varepsilon_{0|x})\|^2 + \tfrac{1}{2}\|L_x(x - x_0)\|^2,
    \label{eq:MAPBAE}
\end{equation}
\begin{equation}
    \Gamma_{\rm post}^{\rm BAE} \approx \bigl(\bar{G}^T \Gamma_{\eta|x}^{-1} \bar{G} + \Gamma_x^{-1}\bigr)^{-1},
    \label{eq:PostCovBAE}
\end{equation}
where $\bar{G} = G + \Gamma_{\varepsilon x}\,\Gamma_x^{-1}$ and $G = D_x g(x^{\rm MAP}_{\rm BAE},t,A(t))$. Note that $G$ is available in closed form for the surrogate, eliminating the need for finite-difference approximations.

The BAE-corrected surrogate $g$ is used in the BAE-corrected Bayesian inference to provide $x_{\rm MAP}^{(k)}$ that drives the control optimisation~\eqref{eq:control_opt}; the control surrogate $h$ is used directly in that optimisation. The accuracy of both surrogates, and of the BAE correction, is assessed empirically for the case study of Section~\ref{sec:Fork}.

\subsection{Surrogate models and the average-pressure parameterisation}
\label{sec:Surrogate}

As mentioned in the previous section, two surrogate models are needed to make the online control tractable. Both of these surrogates can be trained offline, before any data is received. The first, $g$, approximates the pressure forward operator,
\begin{equation}
    g\bigl(t,\, x,\, A(t)\bigr) \approx f\bigl(x,\, t;\, A(t)\bigr),
    \label{eq:Surrogate}
\end{equation}
mapping the fill time $t$, the race-tracking parameters $x$, and the injection pressure history $A(t)$, to predicted sensor pressures $y \in \mathbb{R}^m$. The second surrogate, $h$, maps the same state, along with a candidate future injection pressure $a_f$ to the scalar control objective,
\begin{equation}
    h\bigl(t,\, x,\, A(t),\, a_f\bigr) \approx \Psi.
    \label{eq:ControlSurrogate}
\end{equation}
This surrogate approximates the terminal unfilled measure $\Psi$ of~\eqref{eq:ControlObjectiveDef}, which is minimised in order to determine the optimal control action.

With the two surrogate models defined as above, they both require the injection pressure history $A(t)$ as an input. When discretised this becomes the sequence $a_1, a_2, \dots, a_k$, whose length grows as filling progresses. This is a variable-length input many surrogate models cannot process directly. For the surrogate architectures that can, processing the high-dimensional vector severely inhibits the training and accuracy of the surrogate, an idea known as the curse of dimensionality. In order to make the surrogate training possible, it would be useful to compress $A(t)$ into a fixed, low-dimensional summary without any loss of information about the resin domain $\Ob(t)$. 

The following property of this BVP provides a key insight which we exploit to allow for $A(t)$ to be represented by only a scalar variable per injection gate, but retain all the key information: \textit{Consider an injection-pressure policy $A(t)$ that produces resin domain $\Omega(t)$ at time $t$. Then, an injection-pressure policy $A^*(t)$ that holds the average injection pressure in $A(t)$, $\bar{a}(t) \in \mathbb{R}^{n_g} = \tfrac{1}{t}\int_0^t a(\tau)\,d\tau$, for time $t$ will produce the same resin domain $\Omega(t)$.} The full theorem and proof of this are provided in Appendix~\ref{app:proof}. This establishes that for quasi-static Darcy flow, the filled region $\Omega(t)$ at any fixed time $t$ is determined only by the time-averaged gate pressures vector $\sigma(t)$, not the particular injection history $a(\cdot)$ that produces this average. The computational consequences of this result are significant here, as the resin flow front at time $t$ can be uniquely determined only by the average injection pressures $\bar{a}(t)$. For a given flow front location, we require only the current injection pressures, $a(t)$, at time $t$ to evaluate the pressure field at that time. With this result, the model $f$ does not need the entire injection history $A(t)$ to be uniquely determined. Instead, $f$ can accept $\bar{a}(t)$ and $a(t)$. With this change, $f$ now accepts two low-dimensional, fixed-length vectors (or scalars, if a single injection gate is used), rather than the variable-length, high-dimensional $A(t)$. The surrogate $g$ is reduced to a tractable, fixed-dimensional network $g(t, x, \bar{a}, a_c)$.

Similarly, the dry-area model $\Psi$ needs only the average injection pressure $\bar{a}(t)$ to establish flow front location at time $t$. Here, we see an additional benefit of the average-injection theorem, when looking to the future control actions. Corollary~\ref{cor:constant} in Appendix~\ref{app:proof} establishes the following: \textit{starting from a resin domain $\Omega(t)$, if some admissible time-varying injection policy on $[t, T]$ first reaches the vent at time $T$ with terminal resin domain $\Omega^*(T)$, then the constant policy equal to its time average over $[t, T]$ is admissible, reaches the vent at the same time $T$, and produces the same terminal resin domain $\Omega^*(T)$, and hence the same terminal dry area.}

This means that limiting our search to only constant future injection pressures (as in \eqref{eq:control_opt}) is not an approximation or limitation, but rather considers all potential filling outcomes. As such, the control model $\Psi$ needs only $\bar{a}(t)$ and $a_f$, the constant-pressure injection candidate for the rest of filling. The surrogate $h$ is hence reduced to a tractable, fixed-dimensional network $h(t, x, \bar{a}, a_f)$.

To illustrate the average-pressure parameterisation, we consider a simple two-gate geometry consisting of a rectangle with pressure injection gates at the bottom two corners. We compare two injection policies over a period of 15 seconds.
\begin{itemize}
    \item Policy 1: The right gate is fixed at 100kPa, while the left gate is fixed at 20kPa for 7.5 seconds, then increased to 180kPa for 7.5 seconds.
    \item Policy 2: The left gate is fixed at 100kPa, while the right gate is fixed at 20kPa for 7.5 seconds, then increased to 180kPa for 7.5 seconds.
\end{itemize}
By construction, both reach the same time-averaged gate pressures, $100$~kPa at each gate, by the final time. Figure~\ref{fig:AveragePressure} shows the resulting flow fronts at three fill times for two permeability scenarios: a homogeneous preform and race tracking along both side walls. At early times the running time-averaged pressures of the two policies differ, and the flow fronts differ accordingly. Once the running averages become equal, the fronts, and hence the resin domain $\Ob(t)$, become indistinguishable.

\begin{figure}[htbp]
    \centering
    \includegraphics[width=0.85\textwidth]{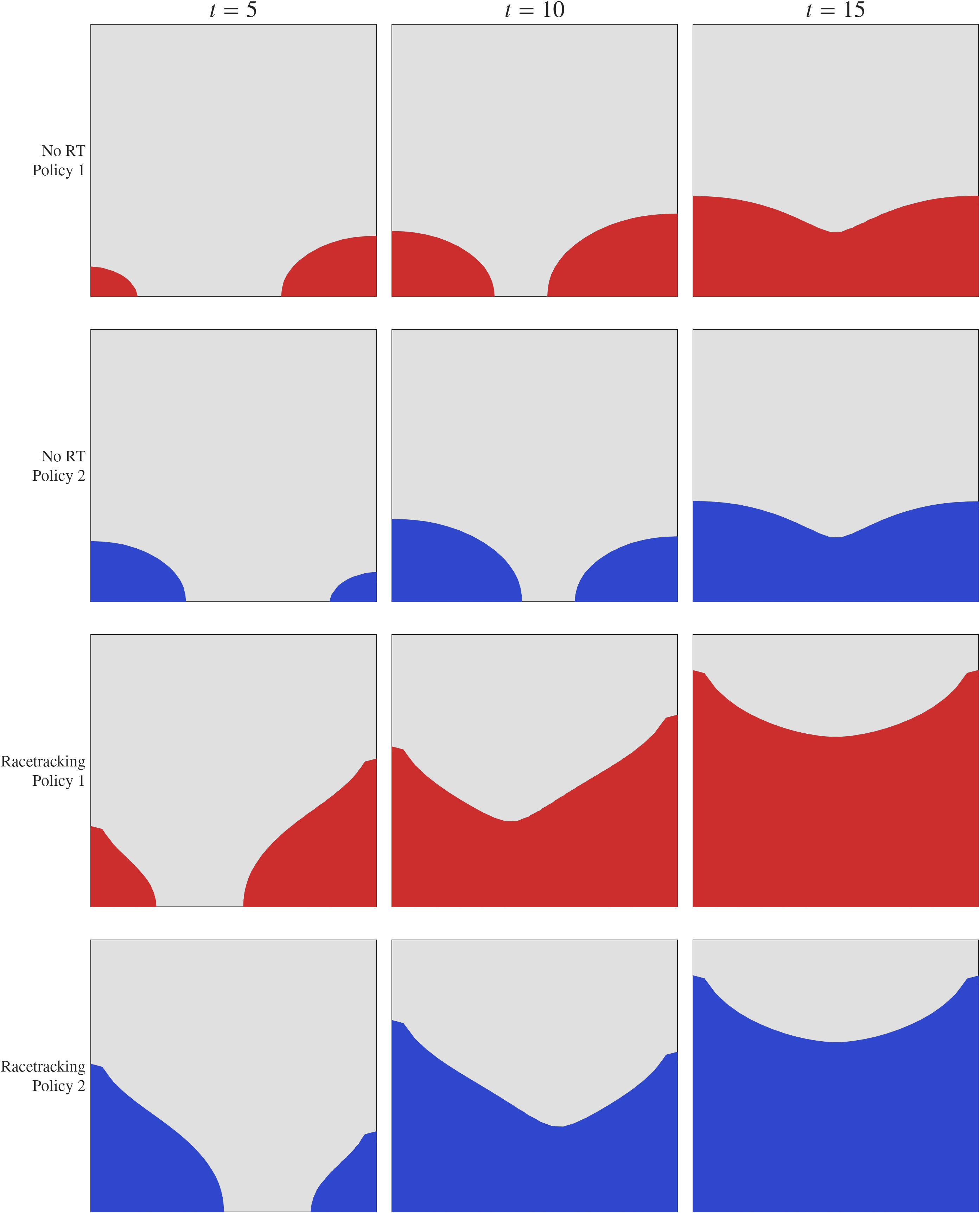}
    \caption{Flow fronts (filled regions) at three fill times ($t = 5,\,10,\,15$; columns) for two injection policies and two permeability scenarios. Policy 1 holds the right gate at $100$kPa while the left gate starts at $20$kPa and switches to $180$kPa at 7.5 seconds. Policy 2 does the opposite. The first two rows correspond to a homogeneous preform (No RT); the last two rows to race tracking along both side walls (Racetracking). At $t = 5$ and $t = 10$ the two policies have different running time-averaged gate pressures and produce different, mirror-image fronts; by $t = 15$ their time-averages coincide and the resin domain $\Ob(t)$, becomes identical.}
    \label{fig:AveragePressure}
\end{figure}

\subsection{The complete online algorithm}
\label{sec:FullAlgorithm}

With the surrogates $g$ and $h$ defined (Section~\ref{sec:Surrogate}) and the average-pressure result established (Theorem~\ref{thm:avgpressure}), the abstract Algorithm~\ref{alg:online_iekf_control} can now be made fully concrete. The two abstract models are replaced as follows.
\begin{itemize}
    \item \textbf{Estimation surrogate.} The pressure surrogate replaces the high-fidelity model in the IEKF estimation step, $f(x,t,\bar{a},a_{\rm current}) \to g(t,x,\bar{a},a_{\rm current})$. The BAE correction (Section~\ref{sec:BAE}) accounts for the discrepancy $\varepsilon(x)$ by correcting the data with the approximation-error mean $\varepsilon_{0|x}$ and the noise covariance with the approximation-error covariance $\Gamma_{\varepsilon|x}$.
    \item \textbf{Control surrogate.} The control surrogate replaces the abstract control model, $\Psi(x,t,\bar{a},a) \to h(t,x,\bar{a},a_{\rm future})$. At each timestep, we find the injection pressure(s) that minimise $h$. In this work, we limit ourselves to the case where we have only one adjustable-pressure injection gate, so this minimisation is achieved with a golden section search over the feasible pressures.
\end{itemize}
The average pressure is updated incrementally after each applied action. The full algorithm is presented in Algorithm~\ref{alg:online_iekf_control_full}:

\begin{algorithm}[h!]
\caption{Online IEKF estimation and receding-horizon control (surrogate-based)}
\label{alg:online_iekf_control_full}
\begin{algorithmic}[1]
\Require Prior mean and covariance $(x_{\rm MAP}^{(0)},\Gamma_x^{(0)})$, surrogates $g,h$, BAE statistics $(\varepsilon_0,\Gamma_\varepsilon,\Gamma_{\varepsilon x})$, feasible set $\mathcal{A}$
\For{$k=1,2,\dots$ until filling completes}
    \State Acquire sensor data $y_k$ at time $t_k$
    \State Compute $x_{\rm MAP}^{(k)}$ with IEKF using $g$ and BAE correction (Gauss--Newton, initialised at $x_{\rm MAP}^{(k-1)}$)
    \State Compute covariance $\Gamma_x^{(k)}$ from linearisation of the BAE-corrected model at $x_{\rm MAP}^{(k)}$
    \State Compute control action $a_k^\star \gets \arg\min_{a\in\mathcal{A}}\, h(x_{\rm MAP}^{(k)},\,t_k,\,\bar{a}_k,\,a)$
    \State Apply $a_k^\star$ on $[t_k,t_{k+1})$
    \State Update $\bar{a}_{k+1} \gets \dfrac{t_k\,\bar{a}_k + (t_{k+1}-t_k)\,a_k^\star}{t_{k+1}}$
\EndFor
\end{algorithmic}
\end{algorithm}

\begin{figure}[h!]
\centering
\begin{tikzpicture}[
    dataflow/.style={draw=blue!60, line width=1.2mm,
        -{Triangle[length=4mm,width=4mm]}},
    trainflow/.style={draw=yellow!80!orange, line width=1.2mm,
        -{Triangle[length=4mm,width=4mm]}},
    input/.style={rounded corners, fill=gray!15, draw,
        text width=2.3cm, align=center, minimum height=1.1cm, font=\small},
    process/.style={rounded corners, fill=blue!15, draw,
        text width=2.4cm, align=center, minimum height=1.1cm, font=\small},
    nn/.style={circle, fill=yellow!20, draw, minimum size=2.6cm},
    approx/.style={rounded corners, fill=yellow!20, draw,
        text width=6.8cm, align=center, minimum height=1.1cm, font=\small},
    output/.style={rounded corners, fill=blue!15, draw,
        text width=3.0cm, align=center, minimum height=1.1cm, font=\small},
    train/.style={rounded corners, fill=blue!70!cyan, draw,
        text width=1.9cm, align=center, minimum height=1.2cm,
        font=\small\bfseries, text=white},
]

\node[input] (prior)   at (-4,  3) {$x_{\rm MAP}^{(0)} = x_0$\\$\Gamma_x^{(0)}=\Gamma_x$};
\node[process] (data)  at (0,  3) {New data\\$y_k$};
\node[process] (IEKF)  at (0,  1) {IEKF update\\$x_{\rm MAP}^{(k)},\Gamma_x^{(k)}$};
\node[process] (control) at (4, 1) {Control\\ optimisation\\$a_k^\star$};
\node[output] (apply) at (4, 3) {Apply $a_k^\star$, update $\bar{a}_{k+1}$};

\draw[<-] (2,3.2) -- (2,3.75);
\node at (2,4) {Increment $k$};

\draw[dataflow] (prior.east) -- (data.west);
\draw[dataflow] (data.south) -- (IEKF.north);
\draw[dataflow] (IEKF.east) -- (control.west);
\draw[dataflow] (control.north) -- (apply.south);
\draw[dataflow] (apply.west) -- (data.east);

\node[approx, anchor=south] (ginfo) at (-6,-2) {\begin{itemize}[leftmargin=*]
    \item NN surrogate $g(t,x,\bar{a},a_c)$
    \item Bayesian inference with BAE correction
    \item Gauss--Newton iterations for MAP estimation
    \item Gaussian approximation of posterior
\end{itemize}};

\node[approx, anchor=south] (hinfo) at (5,-2) {\begin{itemize}[leftmargin=*]
    \item NN surrogate $h(t,x,\bar{a},a_f)$
    \item Golden-section search over $\mathcal{A}$
    \item Receding-horizon control strategy
\end{itemize}};

\draw[trainflow] (IEKF.south) |- (ginfo.east);
\draw[trainflow] (control.east) -| ($(hinfo.north) + (1,0)$);
\end{tikzpicture}
\caption{Schematic of the online estimation and control loop. At each measurement time, new sensor data $y_k$ is used by the IEKF (with surrogate $g$ and BAE correction) to update the MAP estimate $x_{\rm MAP}^{(k)}$ and posterior covariance $\Gamma_x^{(k)}$. The updated estimate drives the control optimisation via surrogate $h$, producing the gate pressure $a_k^\star$ applied until the next measurement. The running average $\bar{a}$ is updated after each step. Blue arrows indicate the loop process, while yellow arrows indicate further details about a particular step.}
\label{fig:loop}
\end{figure}
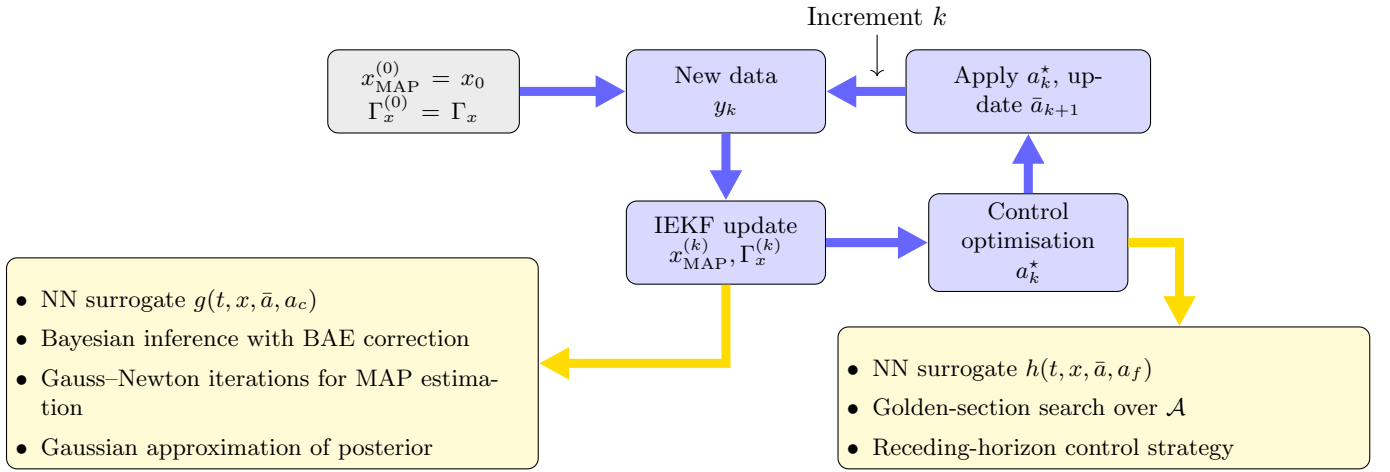

As noted in Section~\ref{sec:ControlObjective}, the objective~\eqref{eq:control_opt} carries no penalty on the change in control effort ($\Delta a$). In an unrestrained system this could in principle lead to undesirable high-frequency \emph{chattering}, or ``slamming'' between the maximum and minimum injection pressures. Under the receding-horizon strategy, however, the race-tracking estimates $x_{\rm MAP}^{(k)}$ tend to evolve smoothly after the first few measurements. While continuity of $h$ alone does not guarantee that its minimiser cannot jump abruptly when $x_{\rm MAP}^{(k)}$ changes slightly, stable actions were observed in our experiments once the estimates had stabilised.

An important caveat is that we use only the MAP estimate when selecting the optimal control action, rather than accounting directly for the full posterior. We choose this approach for simplicity and computational efficiency. This raises a natural question: why do we still carry the posterior covariance $\Gamma_x^{(k)}$ forward at each step? Although the covariance is not explicitly used when selecting the control action, it is used in the next IEKF update as the prior covariance, encoding the cumulative effect of all previous measurements on the race-tracking estimate without requiring those measurements to be stored explicitly.

\section{Numerical Example}
\label{sec:Fork}

To demonstrate the proposed online estimation and control framework, we consider a simple fork geometry with $n=6$ potential race-tracking regions, one independently controllable auxiliary gate ($n_g=1$; the other gate has a fixed injection pressure), and $m=12$ sensors.

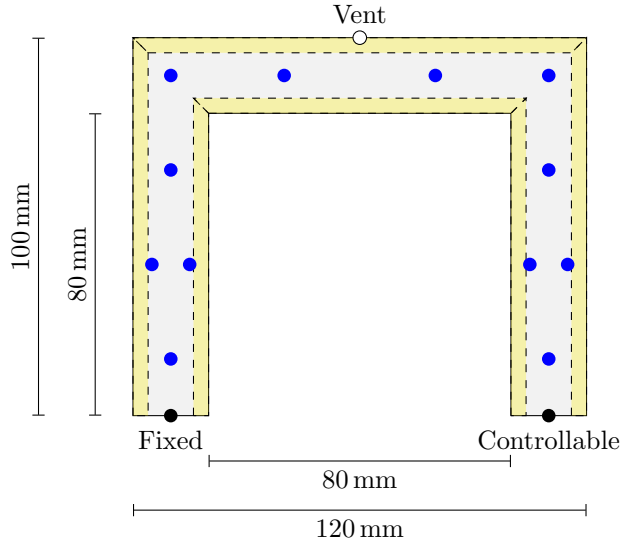
\begin{figure}[h!]
\centering
\begin{tikzpicture}[x=1mm,y=1mm, scale=0.5]
    \draw[fill=gray!10] (0,0) -- (0,100) -- (120,100) -- (120,0) -- (100,0) -- (100,80) -- (20,80) -- (20,0) -- cycle;

    \draw[fill=yellow, fill opacity=0.3, draw=black, dashed] (0,0) -- (0,100) -- (4,96) -- (4,0) -- cycle;
    \draw[fill=yellow, fill opacity=0.3, draw=black, dashed] (16,0) -- (16,84) -- (20,80) -- (20,0) -- cycle;
    \draw[fill=yellow, fill opacity=0.3, draw=black, dashed] (100,0) -- (100,80) -- (104,84) -- (104,0) -- cycle;
    \draw[fill=yellow, fill opacity=0.3, draw=black, dashed] (116,0) -- (116,96) -- (120,100) -- (120,0) -- cycle;
    \draw[fill=yellow, fill opacity=0.3, draw=black, dashed] (4,96) -- (116,96) -- (120,100) -- (0,100) -- cycle;
    \draw[fill=yellow, fill opacity=0.3, draw=black, dashed] (20,80) -- (100,80) -- (104,84) -- (16,84) -- cycle;

    \node[circle, fill=blue, inner sep=1.8pt] at (10,15) {};
    \node[circle, fill=blue, inner sep=1.8pt] at (5,40) {};
    \node[circle, fill=blue, inner sep=1.8pt] at (15,40) {};
    \node[circle, fill=blue, inner sep=1.8pt] at (10,65) {};
    \node[circle, fill=blue, inner sep=1.8pt] at (10,90) {};

    \node[circle, fill=blue, inner sep=1.8pt] at (110,15) {};
    \node[circle, fill=blue, inner sep=1.8pt] at (105,40) {};
    \node[circle, fill=blue, inner sep=1.8pt] at (115,40) {};
    \node[circle, fill=blue, inner sep=1.8pt] at (110,65) {};
    \node[circle, fill=blue, inner sep=1.8pt] at (110,90) {};

    \node[circle, fill=blue, inner sep=1.8pt] at (40,90) {};
    \node[circle, fill=blue, inner sep=1.8pt] at (80,90) {};

    \node[circle, fill=black, inner sep=1.8pt] (fixed) at (10,0) {};
    \node[below=2pt] at (10,0) {Fixed};

    \node[circle, fill=black, inner sep=1.8pt] (ctrl) at (110,0) {};
    \node[below=2pt] at (110,0) {Controllable};

    \node[circle, draw=black, fill=white, inner sep=1.8pt] (ctrl) at (60,100) {};
    \node[above=2pt] at (60,100) {Vent};

    \draw[|-|] (0,-25) -- node[below]{120\,mm} (120,-25);
    \draw[|-|] (-25,0) -- node[above, rotate=90]{100\,mm} (-25,100);
    \draw[|-|] (20,-12) -- node[below]{80\,mm} (100,-12);
    \draw[|-|] (-10,0) -- node[above, rotate=90]{80\,mm} (-10,80);
\end{tikzpicture}
\caption{Fork geometry for the numerical examples illustrated in this section. The fixed and controllable pressure injection gates are annotated, along with the single outlet vent. Race-tracking regions are indicated in yellow, and pressure sensors in blue.}
\label{fig:fork_staple}
\end{figure}

The fork geometry is shown in Figure~\ref{fig:fork_staple}. The two channels are connected by a horizontal section at the top, with a single outlet vent located at the middle of the top edge. The left channel has a fixed injection pressure of $10^5$~Pa (100~kPa), while the right channel has an independently controllable gate. If race tracking occurs in one of the regions on the left side of the fork, the controllable gate should increase its pressure so that the two flow fronts meet in the middle, where the outlet is located. If race tracking is stronger on the right side, the pressure should decrease for the same reason.

Throughout our numerical experiments, all pressure values are given in units of Pascals (Pa). To account for realistic measurement uncertainty, the measurement noise covariance matrix $\Gamma_e$ is taken to be diagonal with all entries set to $10^{6}~\text{Pa}^2$ (corresponding to a standard deviation of $1~\text{kPa}$ on each sensor).

The geometry is meshed into triangular elements with 1618 nodes and 2934 elements for simulation through LIMS. During filling, one pressure measurement is acquired from each sensor every second (a $1$~Hz measurement rate), and the auxiliary gate pressure is re-optimised at each measurement over the feasible set $\mathcal{A} = [0, 200]~\text{kPa}$.

\subsection{Neural network architectures and training}
\label{sec:NN}

As described in Section~\ref{sec:Surrogate}, two neural network surrogates are trained offline: $g$ for inference and $h$ for control. As shown in Figure~\ref{fig:NNtraining}, the training data for both surrogates is generated by evaluating the FECV model on an ensemble of samples drawn from the prior distribution over $x$ and a range of gate-pressure schedules. In addition, our ensemble contains all the solve times from each LIMS run. For the control actions, the samples are generated by taking uniform, independent random values between $0$kPa and $200$kPa. For the control surrogate $h$, the same samples can be used. For each LIMS solve time $t$, we average the elements of $a$ for times greater than $t$. This is the sampled $a_f$, and the terminal sum of the inverse saturations is the corresponding $\Psi$. An important note here is that each LIMS forward model produces many training samples, as each solve time is used as a separate piece of training data for each network.

\begin{figure}[h]
\centering
\begin{tikzpicture}[
    dataflow/.style={draw=blue!60, line width=1.2mm,
        -{Triangle[length=4mm,width=4mm]}},
    trainflow/.style={draw=yellow!80!orange, line width=1.2mm,
        -{Triangle[length=4mm,width=4mm]}},
    input/.style={rounded corners, fill=gray!15, draw,
        text width=2.3cm, align=center, minimum height=1.1cm, font=\small},
    process/.style={rounded corners, fill=blue!15, draw,
        text width=2.4cm, align=center, minimum height=1.1cm, font=\small},
    nn/.style={circle, fill=yellow!20, draw, minimum size=2.6cm},
    approx/.style={rounded corners, fill=yellow!20, draw,
        text width=2.8cm, align=center, minimum height=1.1cm, font=\small},
    output/.style={rounded corners, fill=blue!15, draw,
        text width=3.0cm, align=center, minimum height=1.1cm, font=\small},
    train/.style={rounded corners, fill=blue!70!cyan, draw,
        text width=1.9cm, align=center, minimum height=1.2cm,
        font=\small\bfseries, text=white},
]


\node[input] (x)    at (0,  4.2) {$x \sim \pi(x)$\\[1pt] RT strengths};
\node[input] (a)    at (0,  2.5) {$a(t)$ injection\\[1pt] pressures};
\node[input] (abar) at (0,  0.8) {$\bar{a}$, average\\[1pt] injection pressure};
\node[input] (afut) at (0, -1.3) {$a_{\mathrm{future}}$\\[1pt] future actions};
\node[input] (tgen) at (0, -3.0) {$t$\\[1pt] change times};

\node[process] (lims) at (4, 0.8) {LIMS\\[1pt] $x,a(t)\!\to\!t,p,\Psi$};
\node[output] (t)       at (7.6,  0.8) {$t$\\[1pt] Event times};

\node[nn] (g) at (7.6,  4.2) {};
\node[font=\small] at (7.6, 3.4) {$g(x)$};

\node[nn] (h) at (7.6, -2.6) {};
\node[font=\small] at (7.6, -3.4) {$h(x)$};

\foreach \nn in {g, h}{
    \foreach \y in {0.42, 0, -0.42}{
        \fill[black!70] ($(\nn)+(-0.55,\y)$) circle (0.055);
        \fill[black!70] ($(\nn)+(0.00, \y)$) circle (0.055);
        \fill[black!70] ($(\nn)+(0.55, \y)$) circle (0.055);
    }
    \foreach \ya in {0.42,0,-0.42}{
        \foreach \yb in {0.42,0,-0.42}{
            \draw[black!55, line width=0.28pt]
                ($(\nn)+(-0.55,\ya)$)--($(\nn)+(0,\yb)$);
            \draw[black!55, line width=0.28pt]
                ($(\nn)+(0,\ya)$)--($(\nn)+(0.55,\yb)$);
        }
    }
}

\node[approx] (gapprox) at (11.4,  4.2) {$g(x,t,\bar{a},a_{\mathrm{current}})$\\[1pt] $\approx p$};
\node[output]  (p)       at (11.4,  2.3) {$p$\\[1pt] Pressures at event times};

\node[approx] (happrox) at (11.4, -2.6) {$h(x,t,\bar{a},a_{\mathrm{future}})$\\[1pt] $\approx \Psi$};
\node[output]  (da)      at (11.4, -0.7) {$\Psi$\\[1pt] Dry region after filling};

\node[train] (traing) at (14.7,  3.25) {Train $g(x)$};
\node[train] (trainh) at (14.7, -1.65) {Train $h(x)$};


\draw[trainflow] (x.north) -- ++(0,0.25) -- ++(3.35,0) |- (g.west);
\draw[trainflow] (x.west) -- ++(-0.6,0) -- ++(0,-8.2) -- ($(h.west)+(-3.0,-1.4)$) |- (h.west);
\draw[dataflow] (x.east) -- ++(0.5,0) |- (lims.west);
\draw[dataflow] ($(a.east) + (0,-0.1)$) -- ++(0.5,0) |- (lims.west);
\draw[trainflow] ($(a.east) + (0,0.1)$) -- ++(2.8,0) -- ++(0,0.55) |- (g.west);
\draw[dataflow] (a.south) -- (abar.north);
\draw[trainflow] (abar.east) -- ++(0.15,0) -- ++(0,0.75) -- ++(2.65,0) |- (g.west);
\draw[trainflow] (abar.south) -- ++(0,-0.25) -- ++(2.5,0) |- (h.west);
\draw[trainflow] (afut.east) -- ++(1.21,0) |- (h.west);
\draw[trainflow] (tgen.east) -- ++(0.35,0) |- (h.west);

\draw[dataflow] (lims.east) |- (t.west);
\draw[dataflow] ($(lims.north)+(0.6,0)$) |- (p.west);
\draw[dataflow] ($(lims.south)+(0.6,0)$) |- (da.west);
\draw[trainflow] (t.north) -- (g.south);

\draw[trainflow] (g.east) -- ++(0.28,0) |- (gapprox.west);
\draw[trainflow] (h.east) -- ++(0.28,0) |- (happrox.west);

\draw[trainflow] ($(gapprox.east)+(0,0.28)$) -| (traing.north);
\draw[trainflow] ($(p.east)+(0,-0.32)$) -| (traing.south);
\draw[trainflow] ($(happrox.east)+(0,-0.32)$) -| (trainh.south);
\draw[trainflow] ($(da.east)+(0,0.28)$) -| (trainh.north);

\node[anchor=west, font=\footnotesize] at (0, -4.4) {%
  \tikz[baseline=-0.5ex]{
      \draw[dataflow](0,0)--(0.8,0);}\enspace Data generation\qquad
  \tikz[baseline=-0.5ex]{
      \draw[trainflow](0,0)--(0.8,0);}\enspace Training
};
\end{tikzpicture}
\caption{Schematic of the neural network surrogate training pipeline. The forward model (LIMS) is used to generate training data for both the pressure surrogate $g$ and the control surrogate $h$. The surrogates are trained to predict the pressures at event times and the final dry area, respectively, based on the inputs of the race-tracking parameters, time, and gate-pressure information.}
\label{fig:NNtraining}
\end{figure}
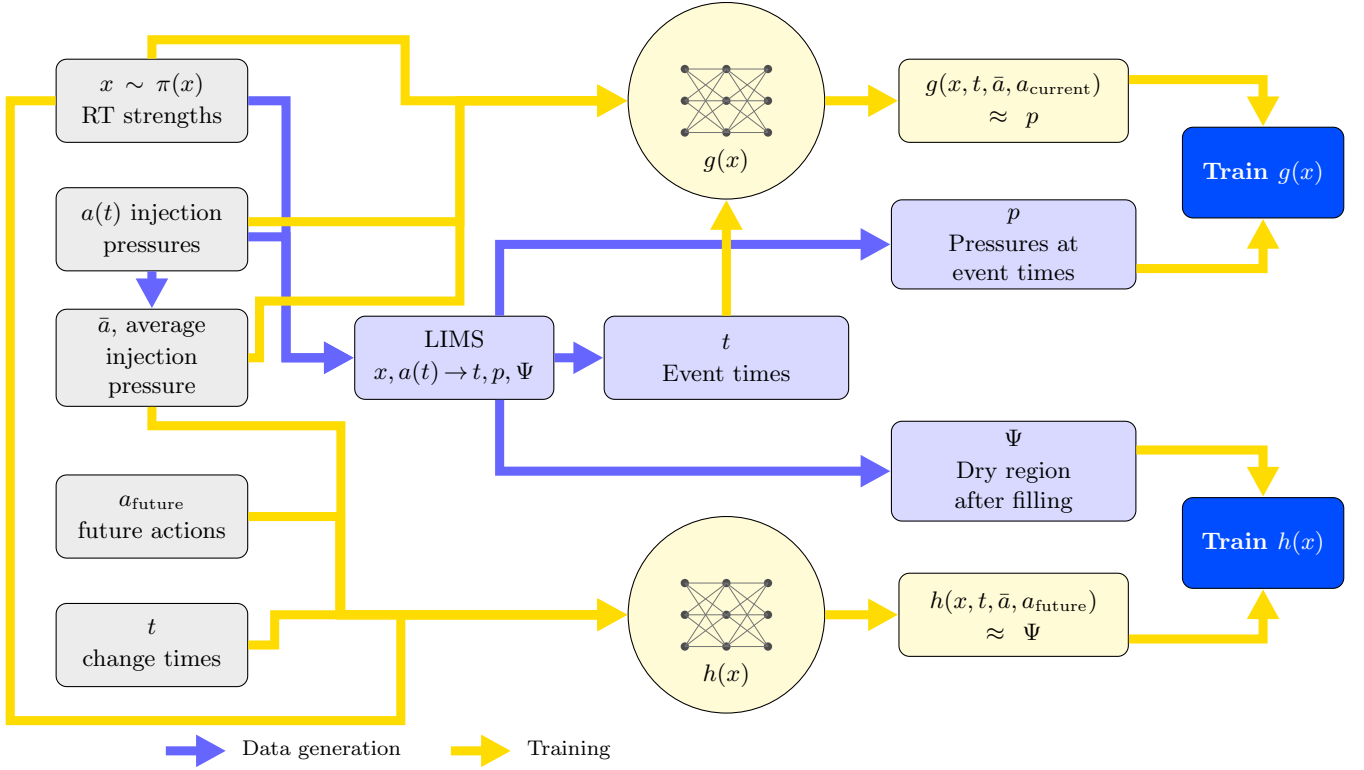

The neural network $g$ maps the input $[x, t, \bar{a}, a_c] \in \mathbb{R}^{n+2n_g+1} = \mathbb{R}^{12}$ and the neural network $h$ maps the input $[x, t, \bar{a}, a_f] \in \mathbb{R}^{n+2n_g+1} = \mathbb{R}^{12}$. Although the networks are distinct, the architecture employed here is similar, only differing in the final output layer to encode their respective output sizes. The input is parsed through three stages of increasing dimension (64, 128, 256). Each stage of the network consists of a linear transformation followed by layer normalization and a nonlinear activation, with dropout applied for regularization. The activations used are standard smooth or piecewise-linear: the Gaussian error linear unit $\mathrm{GELU}(z) = z\,\Phi(z)$, where $\Phi$ is the standard normal cumulative distribution function; the self-gated $\mathrm{Swish}(z) = z\,\sigma(z)$, where $\sigma$ is the logistic sigmoid; and the leaky rectified linear unit $\mathrm{LeakyReLU}(z) = \max(z, 0.01 z)$. A bottleneck residual block (256 $\rightarrow$ 512 $\rightarrow$ 256) is then employed, incorporating a skip connection from the encoder output and using a LeakyReLU activation. The decoder mirrors the encoder by progressively reducing dimensionality (256 $\rightarrow$ 128 $\rightarrow$ 64), and incorporates skip connections from corresponding encoder stages to preserve multi-scale feature information. A final linear layer maps the latent representation to the output of size depending on the network ($m=12$ size in the case of $g$, $1$ in the case of $h$). Full architectural details are provided in Table~\ref{tab:nn_arch}.

\begin{table}[h]
\centering
\caption{Architecture of both surrogate networks. For the pressure surrogate $g$, the encoder--bottleneck--decoder structure takes the input features $[x, t, \bar{a}, a_c] \in \mathbb{R}^{n+2n_g+1}$ and returns the $m$ sensor pressures; the control surrogate $h$ shares the same structure, accepting inputs features $[x, t, \bar{a}, a_f] \in \mathbb{R}^{n+2n_g+1}$ and outputting a single output neuron (the terminal dry area). For each stage we list the layer type, output dimension, activation function, normalization scheme, dropout rate, and any skip or residual connection. The control surrogate uses the same architecture with $m$ replaced by $1$ in the output layer.}
\label{tab:nn_arch}
\begin{tabular}{lllllll}
\toprule
Stage & Layer & Output Dim & Activation & Normalization & Dropout & Skip \\
\midrule
Input  & Feature input & $n+2n_g+1$ & -- & z-score & -- & -- \\

Enc1   & Fully connected & 64  & GELU  & LayerNorm & 0.1  & -- \\
Enc2   & Fully connected & 128 & Swish & LayerNorm & 0.1  & -- \\
Enc3   & Fully connected & 256 & GELU  & LayerNorm & 0.15 & -- \\

Bottleneck & FC $\rightarrow$ FC & 512 $\rightarrow$ 256 & LeakyReLU & LayerNorm & -- & Residual (Enc3) \\

Dec1   & Fully connected & 128 & Swish & LayerNorm & -- & Skip (Enc2) \\
Dec2   & Fully connected & 64  & GELU  & LayerNorm & -- & -- \\

Output & Fully connected & $m/1$ & -- & -- & -- & -- \\
\bottomrule
\end{tabular}
\end{table}

LIMS is run 100,000 times for training-data generation, producing approximately 159 million samples of the form $(x, \bar{a}, a_c, a_f, \Psi, t, p)$, where $\bar{a}$, $a_c$ and $a_f$ are computed from $a(t)$. It is important to emphasize that each of the 100,000 forward runs produces many individual data points - one for each time step that the fluid pressure is solved. An additional 10,000 LIMS simulations are run for testing and validation. In practice, however, we do not train on the full generated dataset: a random 5\% subset of the training data points is used, with random 1\% subsets of the validation and test data points. This downsampling is necessary because the full dataset is too large for efficient end-to-end training, and using random subsets keeps GPU memory usage, data-loading overhead, and epoch times tractable while ensuring we capture a diverse, representative range of race-tracking scenarios. Training is done in two stages using the adaptive moment estimation (Adam) optimizer~\cite{kingma2014adam} and a batch size of 512. In the first stage, training is performed for 3 epochs with a learning rate of $10^{-3}$, and the loss function is the mean squared error (MSE) between the predicted and true sensor pressures at the event times. After the first stage is complete, a ReLU activation layer is added to the output layer to ensure non-negativity of the predicted pressures, and training continues for an additional epoch with a learning rate of $10^{-4}$. Both training loops are allowed to terminate early if there is no validation improvement in 500 iterations. To ensure stable gradients and efficient weight updates, all target pressure values are scaled by a factor of $10^5$~Pa during training. The mean squared errors (MSE) reported during this phase reflect this scaled space. The first round of training produced a scaled validation loss of $10^{-3}$, and after the second round this reduced to $7.5 \times 10^{-4}$. In physical units, this final MSE corresponds to a root mean square error (RMSE) of approximately $2.74 \times 10^3$~Pa (or $2.74$~kPa) per sensor. The training and validation loss over time is visualised in Appendix~\ref{app:hTraining}, Figure~\ref{fig:NNtrainingloss}. This training takes a total of 2100 seconds on a single RTX 5080 NVIDIA GPU.

Training is done in two stages to allow the model to learn the general flow physics first and then refine the predictions to ensure non-negativity of the pressures. Training with the ReLU activation layer from the start can cause the gradients of many network weights and biases to become zero early on, which can prevent the model from converging. By first training without the ReLU, the model learns a good initial representation of the flow physics, after which the ReLU can be added to fine-tune the predictions while preserving the learned representation.

For $h(x)$, we use an identically structured network, but with only one output neuron (representing $\Psi$). There is only one training session with a learning rate of $10^{-3}$, and we train with as many epochs as needed for the validation loss to begin increasing. Training takes approximately 1900s on a single RTX 5080 NVIDIA GPU, and produces a final validation loss of $\approx 155$ (in squared node counts), after approximately 3 epochs. The corresponding training and validation loss curves are shown in Appendix~\ref{app:hTraining}, Figure~\ref{fig:hTrainingLoss}.

Once the surrogates have finished training, we compute the BAE for $g$ as described in Section~\ref{sec:BAE}. We use the testing set of parameters $x$ sampled from the prior, together with randomly generated values for $\bar{a}$ and $a_c$ (sampled uniformly between 0~kPa and 200~kPa, similar to the training data). For these input sets, we evaluate $f$ and $g$ using LIMS and the first surrogate, respectively. We then compute $\varepsilon_{0|x}$ and $\Gamma_{\varepsilon|x}$ from these samples using the standard formulas for Gaussian means and covariances.

\paragraph{Computational performance.}
The surrogates render the online loop real-time capable. A single full-fill LIMS simulation of this fork geometry takes approximately $0.65$~s. Carrying out an IEKF update with the FECV model would require a finite-difference Jacobian ($n+1$ forward solves per linearisation) over five Gauss--Newton iterations, that is, roughly $5(n+1)=35$ full-fill simulations, taking approximately $23$~s per measurement update - far exceeding the one-second interval between measurements. With the surrogate $g$, which supplies a closed-form Jacobian via automatic differentiation, the same update completes in approximately $0.34$~s, so that estimation and control comfortably keep pace with the $1$~Hz measurement rate.

\subsection{Results}
\label{sec:ForkResults}
We define a successful fill as one with terminal dry area (or, more specifically, the proxy $\Psi$) below 2\% of the mesh, i.e., fewer than 2\% of nodes remain empty.
We first consider a manufactured race-tracking scenario in which one race-tracking parameter is dominant ($x_2=4$) and the remaining parameters are zero. In this case, our algorithm quickly identifies the race-tracking within the first few seconds, and increases the auxiliary-gate pressure above the fixed-gate pressure of 100~kPa, thereby rebalancing the advancing fronts and reducing premature vent arrival on the faster side. A representative snapshot is shown in Figure~\ref{fig:fork_manufactured_frames}, where the controlled front is more closely aligned with the outlet region than the uncontrolled counterpart.

\begin{figure}[h!]
\centering

\begin{minipage}[b]{0.24\textwidth}\centering $t = 1\,\mathrm{s}$\end{minipage}\hfill
\begin{minipage}[b]{0.24\textwidth}\centering $t = 20\,\mathrm{s}$\end{minipage}\hfill
\begin{minipage}[b]{0.24\textwidth}\centering $t = 40\,\mathrm{s}$\end{minipage}\hfill
\begin{minipage}[b]{0.24\textwidth}\centering Complete\end{minipage}

\vspace{1mm}

\begin{subfigure}[b]{0.24\textwidth}
    \centering
    \includegraphics[width=\linewidth]{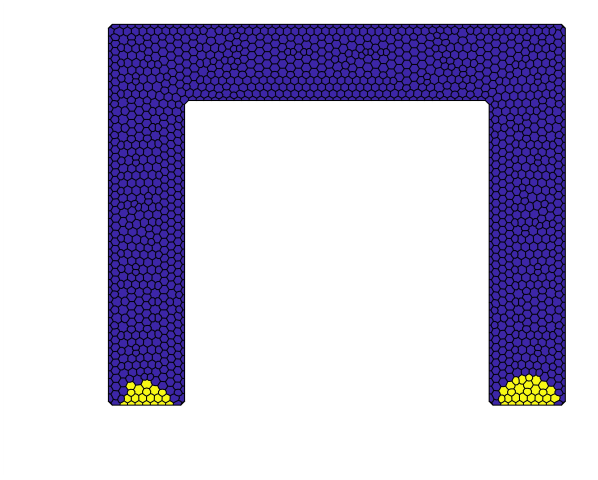}
\end{subfigure}\hfill
\begin{subfigure}[b]{0.24\textwidth}
    \centering
    \includegraphics[width=\linewidth]{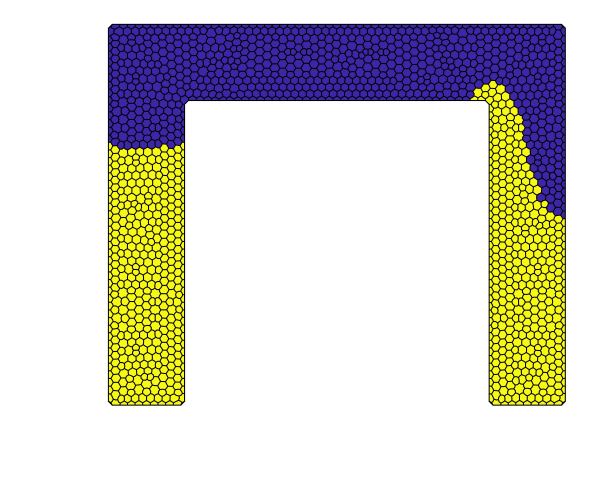}
\end{subfigure}\hfill
\begin{subfigure}[b]{0.24\textwidth}
    \centering
    \includegraphics[width=\linewidth]{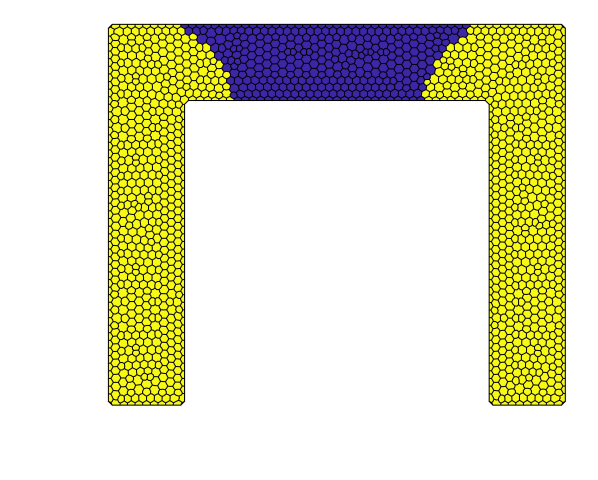}
\end{subfigure}\hfill
\begin{subfigure}[b]{0.24\textwidth}
    \centering
    \includegraphics[width=\linewidth]{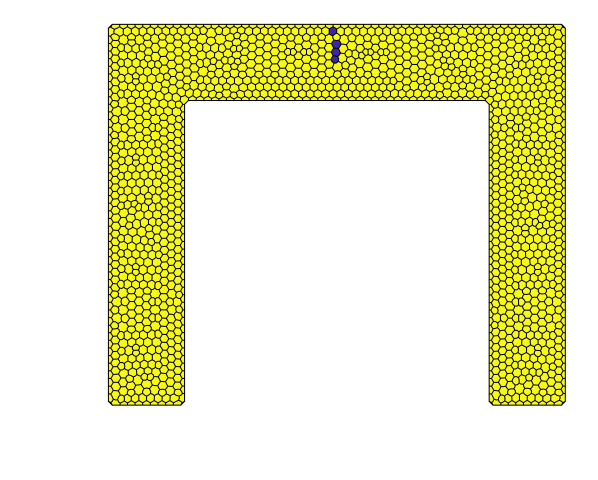}
\end{subfigure}

\vspace{-2mm}

\begin{subfigure}[b]{0.24\textwidth}
    \centering
    \includegraphics[width=\linewidth]{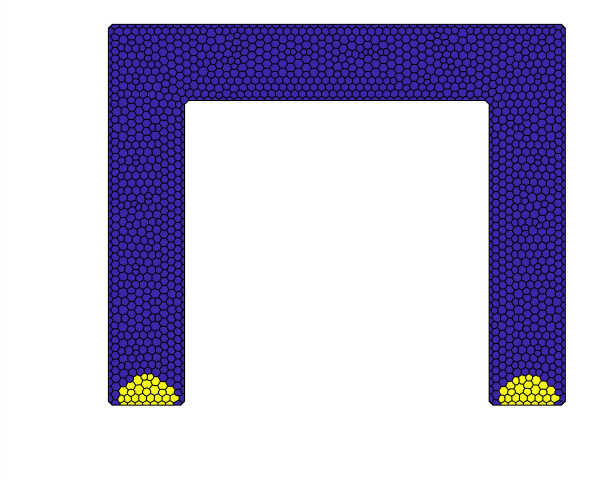}
\end{subfigure}\hfill
\begin{subfigure}[b]{0.24\textwidth}
    \centering
    \includegraphics[width=\linewidth]{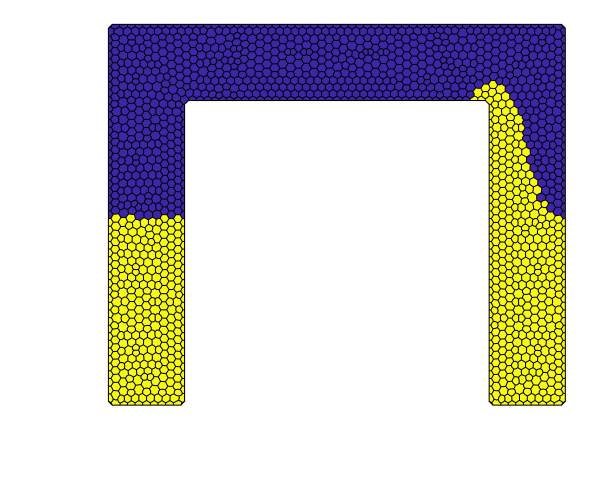}
\end{subfigure}\hfill
\begin{subfigure}[b]{0.24\textwidth}
    \centering
    \includegraphics[width=\linewidth]{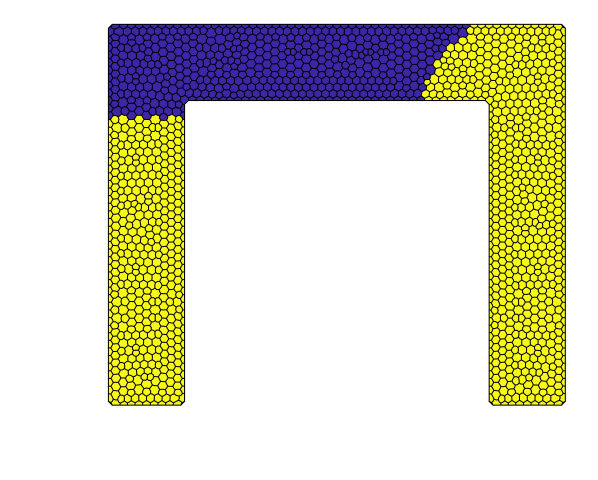}
\end{subfigure}\hfill
\begin{subfigure}[b]{0.24\textwidth}
    \centering
    \includegraphics[width=\linewidth]{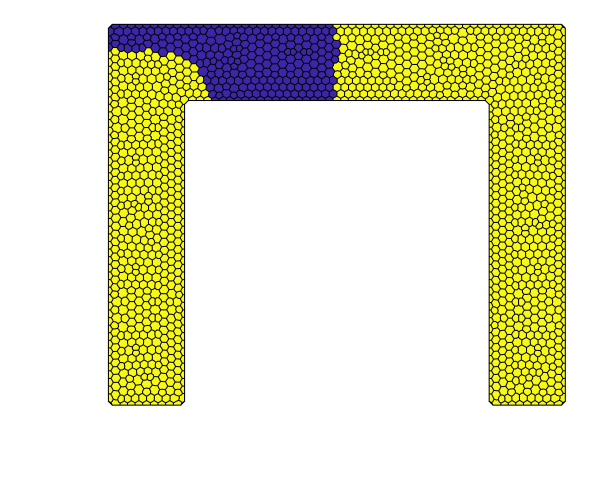}
\end{subfigure}

\vspace{2mm}

\begin{subfigure}[b]{0.24\textwidth}
    \centering
    \includegraphics[width=\linewidth]{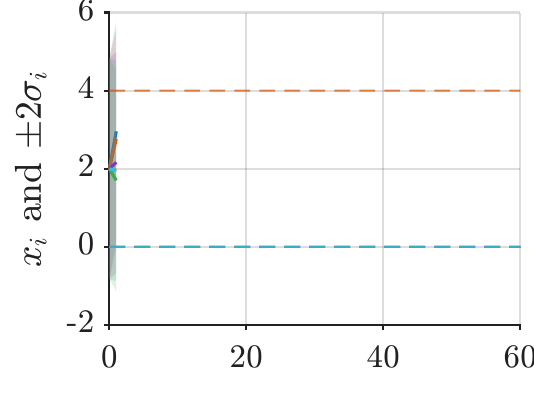}
\end{subfigure}\hfill
\begin{subfigure}[b]{0.24\textwidth}
    \centering
    \includegraphics[width=\linewidth]{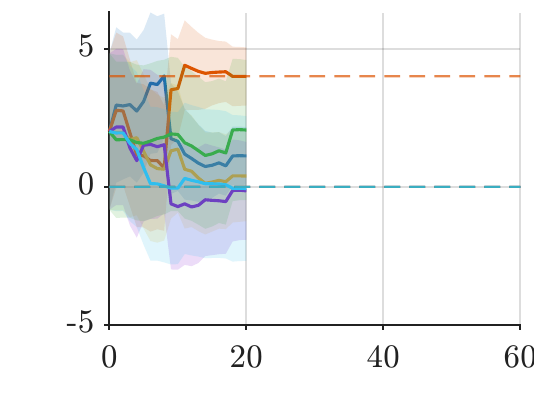}
\end{subfigure}\hfill
\begin{subfigure}[b]{0.24\textwidth}
    \centering
    \includegraphics[width=\linewidth]{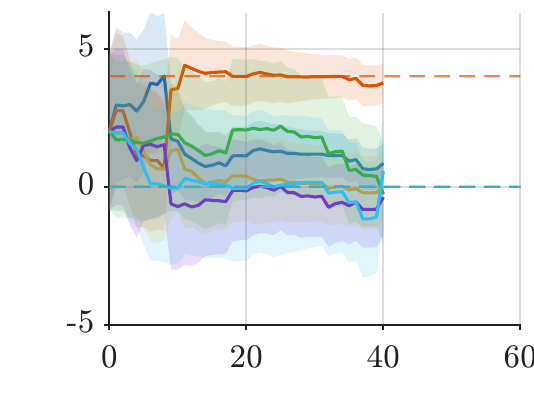}
\end{subfigure}\hfill
\begin{subfigure}[b]{0.24\textwidth}
    \centering
    \includegraphics[width=\linewidth]{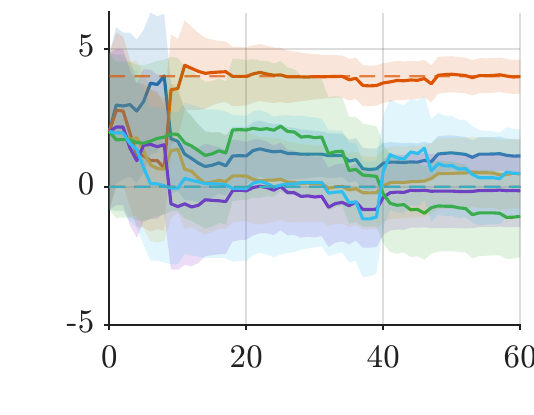}
\end{subfigure}

\vspace{0mm}

\begin{subfigure}[b]{0.24\textwidth}
    \centering
    \includegraphics[width=\linewidth]{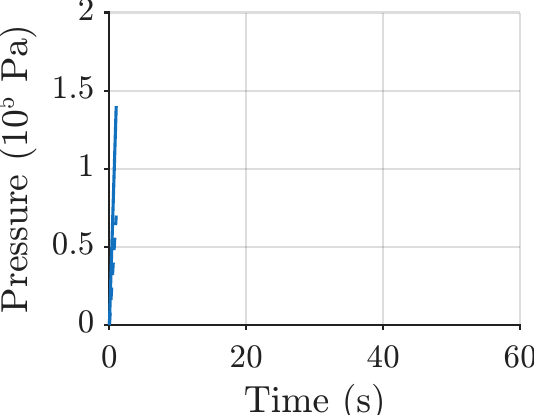}
\end{subfigure}\hfill
\begin{subfigure}[b]{0.24\textwidth}
    \centering
    \includegraphics[width=\linewidth]{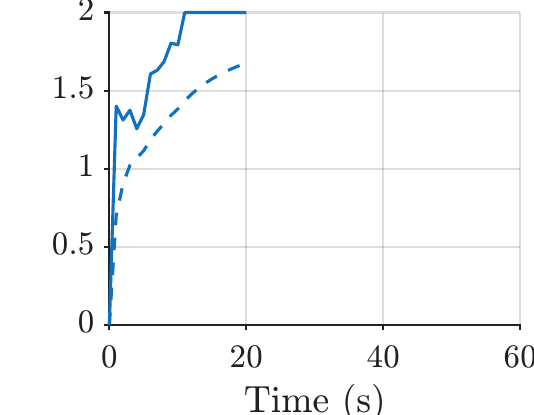}
\end{subfigure}\hfill
\begin{subfigure}[b]{0.24\textwidth}
    \centering
    \includegraphics[width=\linewidth]{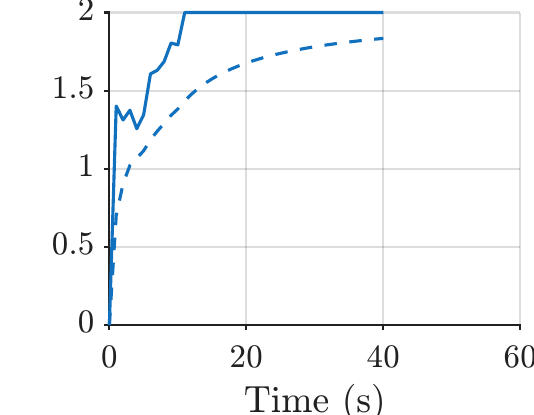}
\end{subfigure}\hfill
\begin{subfigure}[b]{0.24\textwidth}
    \centering
    \includegraphics[width=\linewidth]{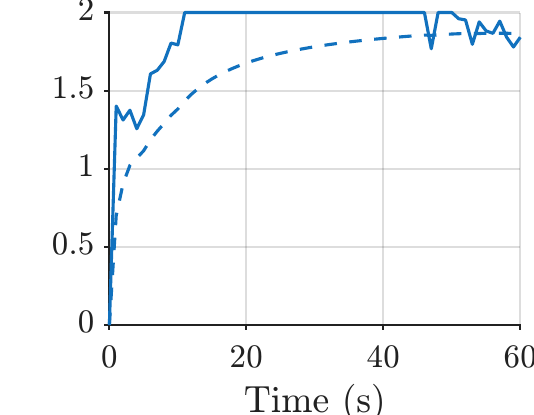}
\end{subfigure}

\caption{
Evolution of the closed-loop controller during the filling process at four representative timesteps (columns). The top row depicts the actively controlled flow front, compared against the nominal, uncontrolled flow front in the second row (auxiliary gate fixed at $100\,\mathrm{kPa}$). The third row illustrates the filter's estimation of the race-tracking variables over time, correctly identifying the active channel (solid lines are means, shaded regions are $\pm 2$ standard deviations, and dashed lines represent the ground truth). The final row shows the gate pressure control commands generated by the network with a solid line and the dashed line shows the time-averaged gate pressure $\bar{a}$. Note that for the bottom two rows, the previous plots are simply subsets of the final column.
}
\label{fig:fork_manufactured_frames}
\end{figure}

To quantify performance at population level, we evaluated 100 Monte Carlo simulations with race-tracking strengths sampled from the prior. A run is deemed successful if the terminal dry area is below 2\% of the mesh (fewer than 33 empty nodes out of 1618). Under this criterion, 18/100 no-control runs are successful, compared with 42/100 when the controlled actions are evaluated with the FECV model $\Psi$. The surrogate-based controller $h$ predicts 58/100 successes, which is greater than the actual number of successes (as expected, since the control surrogate $h$ is optimistic). The terminal dry-node distributions and the per-run improvements are shown in Figure~\ref{fig:fork_batch_combined}.

\begin{figure}[h!]
\centering
\begin{minipage}[t]{0.48\textwidth}
    \centering
    \includegraphics[width=\linewidth]{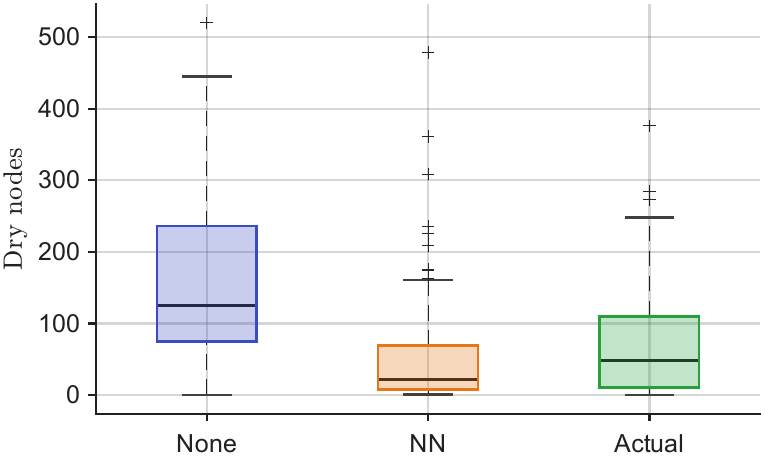}
\end{minipage}\hfill
\begin{minipage}[t]{0.48\textwidth}
    \centering
    \includegraphics[width=\linewidth]{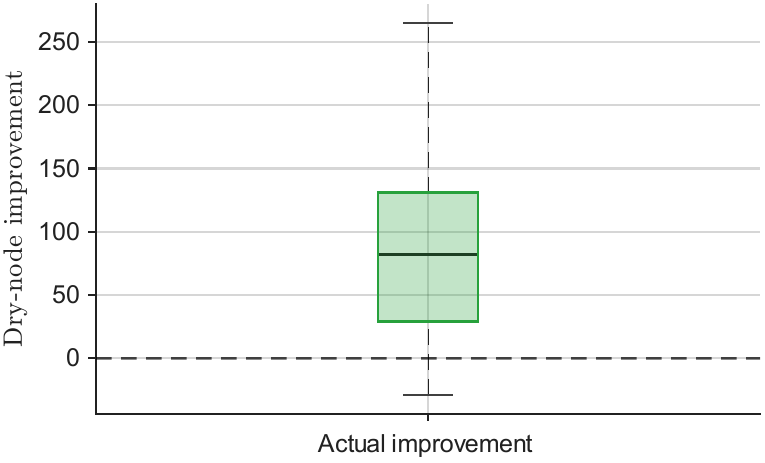}
\end{minipage}
\caption{Dry-node outcomes for the fork geometry across 100 Monte Carlo simulations. Left: terminal dry-node counts at the end of filling for no control (None), surrogate-predicted control (NN), and controlled runs evaluated with the high-fidelity model (Actual); the controlled strategy shifts the distribution downward, with surrogate-based predictions optimistic but directionally consistent with the high-fidelity evaluations. Right: distribution of the per-run improvement in dry-node count (no control minus controlled) across the same 100 simulations, where positive values indicate fewer dry nodes under control.}
\label{fig:fork_batch_combined}
\end{figure}

\subsection{Controllability}

The performance of the control strategy depends strongly on the range of control actions available. In this case, the single auxiliary gate cannot always be adjusted to achieve perfect filling when constrained to 0-200~kPa, so the control strategy cannot guarantee perfect filling in all scenarios. This remains true even when the online Bayesian methodology predicts the race-tracking strengths accurately. In some cases, race tracking occurs along the top of the mould, and the algorithm does not have sufficient time to estimate this effect and adjust the gate pressure before the resin reaches the vent.

To assess the controllability of the system, we consider each race-tracking scenario in turn, trialling many candidate auxiliary gate pressures, each held constant for the entire fill. Restricting attention to constant pressures incurs no loss of generality: by Corollary~\ref{cor:constant} in Appendix~\ref{app:proof}, the terminal configuration reached by any admissible time-varying policy is reproduced by the constant pressure equal to its time average, which reaches the vent at the same instant (the feasible set $[0, 200]$~kPa is convex, as the corollary requires). Consequently, a scenario that cannot be filled by any constant auxiliary pressure cannot be filled by any time-varying policy either, and the constant-pressure scan is a complete controllability test. If the minimum achievable dry area across all these trials is above the 2\% threshold, then we can conclude that the system is not controllable for that scenario. In this case, we find that only 51/100 scenarios were controllable, which is therefore a genuine upper bound on the success rate of any control strategy within the feasible action space. In this light, our increase from 18/100 to 42/100 is more notable, as the control strategy succeeds in 42 of the 51 controllable scenarios, corresponding to an 82\% success rate within the subset of scenarios where control is possible.


\section{Conclusion}
\label{sec:Con}

This paper presents a framework for online control of the RTM process, resolving the computational cost associated with time-varying control histories. By utilising the fact that the flow front location is uniquely determined by only the average injection pressure, we can train neural network surrogate models that emulate the computationally expensive and non-differentiable LIMS FECV model. We combine online Bayesian inference with the fast, differentiable surrogate models, correcting the surrogate--solver discrepancy through the BAE framework, to obtain a real-time-capable estimation and receding-horizon control strategy that reduces the terminal dry area in the numerical fork-geometry study. Under the 2\% dry-area criterion, the success rate rises from 18/100 without control to 42/100 with the controlled strategy. Although this is a substantial improvement, the success rate is still low.

There are two key constraining factors for the success of this algorithm that could be investigated in future work. The first and most important is the controllability of the system. If the control actions are not sufficiently powerful or flexible to rebalance the flow fronts, then even a perfect estimate of the race-tracking parameters will not guarantee a successful fill. In the present study the single auxiliary gate cannot always be adjusted to achieve perfect filling within the feasible action space, and we have shown that in our numerical example, only 51 out of 100 race-tracking scenarios were controllable with the current gate and vent design. This result shows that gate and vent locations (as well as numbers) should be carefully designed to ensure that the system is as controllable as possible, so that as the race-tracking estimates are found, the control actions can be effective. Similarly, optimal sensor placement should be considered to ensure maximum observability and minimum uncertainty in the race-tracking parameter estimates \cite{Wright2023}. This is a promising direction for future research, and there are already a large number of papers focussing on the idea of optimally designing / positioning gates, vents and sensors within the mould~\cite{Young1994, Mathur1999, Ali2002, GokceAdvani2004, Ye2004, Chai2021, NiknafsKermani2026}.

The second constraint is the accuracy of the surrogate models. As the focus of the present work is on demonstrating the overall framework, we have not fully optimised the surrogate architectures or training procedures. Improvements to both the pressure surrogate $g$ and the control surrogate $h$ would improve the performance of the algorithm. A more accurate $g$ reduces the approximation error $\varepsilon(x) = f(x) - g(x)$, yielding a smaller BAE covariance and reduced posterior uncertainty throughout filling, while a more accurate $h$ yields more accurate optimal control actions. Future work could investigate other surrogate architectures, including the use of Gaussian processes or physics-informed neural networks (PINNs)~\cite{Raissi2019, Yu2022, Chiu2022}.

There are two small caveats with our formulation here to note. As the number of controllable gates increases (to improve system controllability), the simple one-dimensional golden-section search used here will not scale, and gradient-based optimisation over the differentiable surrogate $h$ will be required to find the optimal control actions efficiently. Secondly, we note that formulating the control objective as a pure minimisation without penalizing control effort ($\Delta a$) could lead to aggressive pressure changes in an unrestrained physical setup, although this was not observed here. The methodology has so far been demonstrated on a single simple geometry and evaluated numerically only. Future work should consider more complex geometries and, ultimately, physical experiments. The framework is general and applies to any geometry, but its performance will depend on the controllability of the system and the accuracy of the surrogates.

A known concern with sequential Bayesian estimators such as the IEKF is sensitivity to prior misspecification. Previous investigations of this estimation framework have demonstrated that the BAE-corrected framework yields parameter estimates that remain robust even when the initial prior distributions are intentionally misaligned with the true race-tracking parameters \cite{Wright2026}. Consequently, we anticipate that the estimation phase of this algorithm will translate reliably to similarly bounded misalignments, though its reliability under more extreme or novel race-tracking configurations remains an open validation question.

Finally, transitioning from this simulation-based framework to a physical RTM process~\cite{Rudd2001} will introduce the ``sim-to-real'' gap. LIMS models isothermal Darcy flow, whereas real thermosetting resins undergo complex curing kinetics during filling, leading to dynamic changes in viscosity and temperature~\cite{Lee2000}. Validating the control policy against such transient, un-modelled physical effects will be a critical step toward deployment in production~\cite{Lee2024}.

\section*{Declarations}

\paragraph{Funding} The authors have no funding to declare.

\paragraph{Data availability} The data supporting the findings of this study are available from the corresponding author on request.

\bibliographystyle{plain}
\bibliography{bibliography}

\appendix

\section{Proof of the average-pressure theorem}
\label{app:proof}

Here, we provide a proof of the average-injection-pressure parameterisation detailed in Section~\ref{sec:Surrogate}, formalised as Theorem 1 below. The argument follows standard techniques used in free-boundary theory. We refer the reader to~\cite{Baiocchi1972, Duvaut1973, Friedman1982, KinderlehrerStampacchia1980}, among others, for more details.

\begin{theorem}[Average-pressure path-independence]
\label{thm:avgpressure}
Let $D\subset\mathbb{R}^d$ with $d=2$ or $d=3$, be a bounded Lipschitz domain, let $K:D\to\mathbb{R}^{d\times d}$ be measurable, essentially bounded, and uniformly positive definite and assume $\mu,\phi>0$ are constant. Furthermore, let $a\in L^1(0,T_a;\mathbb{R}_+^{n_g})$ and $b\in L^1(0,T_b;\mathbb{R}_+^{n_g})$ denote two non-negative gate pressure histories which induce (through the quasi-static Darcy filling problem (\ref{eq:Darcy})-(\ref{eq:FrontVelocity})) the filled regions $\Omega_a(t)$ and $\Omega_{b}(t)$, respectively, with $T_a$ and $T_b$ denoting the filling times when using $a(t)$ and $b(t)$, respectively. Then for a fixed time $0< t\leq\min\{T_a,T_b\}$, if 
\[\int_0^t a(\tau)\;d\tau=\int_0^t b(\tau)\;d\tau,\]
then the corresponding filled-region indicator functions satisfy
\[\boldsymbol{1}_{\Omega_a(t)}=\boldsymbol{1}_{\Omega_b(t)},\quad a.e. \text{ in } D.\]
That is to say, at time $t$ the filled region depends only on the cumulative gate pressures
\[\mathbb{R}^{n_g}\ni\sigma(t):=\int_0^ta(\tau)\;d\tau=t\bar{a}.\]
\end{theorem}

\begin{proof}
    Assume the assumptions in the theorem hold and define $\Gamma_{\rm s}:=  \Gamma_{\text{vent}} \cup (\cup_{j=1}^{n_g} \Gamma_{\text{in}}^{j})$ and $V_0(D):=\{v\in H^1(D)\;: \left.v\right|_{\Gamma_{\rm s}}\!=0\}$, and for notational simplicity set $\mu=1$.  Then for a given time $t$ before the filling time $T$, i.e., $0<t<T$, denote the filled region at time $t$ by $\Omega(t)$. Let $p$ be the solution to the quasi-static Darcy filling problem (\ref{eq:Darcy})-(\ref{eq:FrontVelocity}). Taking $w\in V_0(D)$ as a test function, we have
    \[0=\int_{\Omega(t)}(K\nabla p) \cdot\nabla w \;ds-\int_{\Gamma_{\text{front}}(t)}(K\nabla p\cdot n) w \;dS=\int_{\Omega(t)}(K\nabla p) \cdot\nabla w \;ds+\phi\int_{\Gamma_{\text{front}}(t)} v_n w \;dS,\]
    where the first equality follows from using integration by parts, and from the fact that all other boundary terms vanish (since $w=0$ on $\Gamma_{\rm s}$ and $q\cdot n=(-K\nabla p)\cdot n=0$ on $\partial D\setminus\Gamma_{\rm s}$), and the last equality follows by plugging in the front velocity (\ref{eq:FrontVelocity}).

    Now, by Reynolds transport theorem (see for example~\cite[Result 4.10]{GonzalezStuart2008}), since the test function $w$ is independent of time, we have
    \begin{align}
        -\phi\frac{d}{dt}\int_{\Omega(t)} w\;ds=-\phi\int_{\Gamma_{\text{front}}(t)} v_n w \;dS=\int_{\Omega(t)}(K\nabla p) \cdot\nabla w \;ds.\label{eq:Reynold}
    \end{align}

    Now, consider extending $p$ by $0$ into $D\setminus\Omega(t)$ and let $\boldsymbol{1}_{\Omega(t)}$ be the filled-region indicator function for $\Omega(t)$. Then we can rewrite (\ref{eq:Reynold}) over the domain $D$ as
    \begin{align}
        \phi\frac{d}{dt}\int_D \boldsymbol{1}_{\Omega(t)}w\;ds+\int_{D}(K\nabla p) \cdot\nabla w \;ds=0.\label{eq:ddtFront}
    \end{align}
     
At this point we employ the Baiocchi transform~\cite{cummings1999two,baiocchi1984variational} to convert the moving-boundary problem into a fixed-domain obstacle problem. Specifically, by introducing the time-integrated pressure
\[u(\spacev,t) := \int_0^t p(\spacev,\tau) \, d\tau,\]
and integrating (\ref{eq:ddtFront}) and the gate-pressures $a(t)$ with respect to $t$ from $0$ to $t$, the time integrated pressure satisfies
\begin{align}
    \phi\int_D (\boldsymbol{1}_{\Omega(t)}-\boldsymbol{1}_{\Omega(0)})w\;ds+\int_{D}(K\nabla u(t)) \cdot\nabla w \;ds&=0,\quad w\in V_0,\label{eq:weakBeq}\\
    u(t)&=\sigma_j(t),\quad \text{on } \Gamma_{\text{in}}^{j},\quad j=1,2,\dots,n_g,\label{eq:weakBeqBC}\\
    u(t)&=0,\quad \text{on } \Gamma_{\rm vent},
\end{align}
with $\int_0^t a_j(\tau)\;d\tau=\sigma_j(t)$, and where the regularity of $K$, $p$ and $w$ allows us to use Fubini's theorem to swap the time and spatial integrals. Notice that the time integrated pressure $u(s,t)$ is always non-negative, since $p$ is non-negative, i.e., $u(s,t)\geq 0$, and that if at time $t$ the point $s$ is dry, it has been dry for all $0\leq\tau\leq t$, meaning $p(s,\tau)=0$ which implies then  $u(s,t)=0$.  Moreover, if the point $s$ is dry at time $t$ then $\boldsymbol{1}_{\Omega(t)}=0$ at $s$ and we have the relationship
\begin{align}
    u(1-\boldsymbol{1}_{\Omega(t)})=0.\label{eq:keyEq}
\end{align}

At this point, following the standard variational formulation of elliptic obstacle problems; see, e.g., \cite{KinderlehrerStampacchia1980}, we let $\zeta\in H^1(D)$ be a non-negative function, i.e., $\zeta\geq0$ a.e. in $D$ which satisfies the same Dirichlet boundary conditions as $u(t)$. Then, picking $w=\zeta-u(t)\in V_0$ in Equation (\ref{eq:weakBeq}) and rearranging, we have
\[\int_{D}(K\nabla u(t)) \cdot\nabla (\zeta-u(t)) \;ds=\phi\int_D (\boldsymbol{1}_{\Omega(0)}-\boldsymbol{1}_{\Omega(t)})(\zeta-u(t))\;ds.\]
Adding $\phi\int_D (1-\boldsymbol{1}_{\Omega(0)})(\zeta-u(t))\;ds$ to each side gives
    \[\int_{D}(K\nabla u(t)) \cdot\nabla (\zeta-u(t)) \;ds+\phi\int_D (1-\boldsymbol{1}_{\Omega(0)})(\zeta-u(t))\;ds=\phi\int_D (1-\boldsymbol{1}_{\Omega(t)})(\zeta-u(t))\;ds.
    \]
    Using (\ref{eq:keyEq}) and the fact that $\phi\int_D (1-\boldsymbol{1}_{\Omega(t)})\zeta\;ds\geq 0$, we have
   \begin{align}\int_{D}(K\nabla u(t)) \cdot\nabla (\zeta-u(t)) \;ds+\phi\int_D (1-\boldsymbol{1}_{\Omega(0)})(\zeta-u(t))\;ds\geq 0\label{eq:Varenq}.
\end{align}
The key insight is that at the fixed time $t$, the filled region $\Omega(t)$ and the moving front $\Gamma_{\rm front}(t)$ do not explicitly appear in \eqref{eq:Varenq}. That is to say, the gate pressure history enters only through the cumulative gate pressures $\sigma(t)$. Hence, for two non-negative gate pressure histories satisfying 
\[\int_0^t a(\tau)\;d\tau=\int_0^t b(\tau)\;d\tau,\]
the resulting time-integrated pressures $u_a(t)$ and $u_b(t)$ satisfy the same fixed domain variational inequality with the same Dirichlet boundary conditions. Then, by uniqueness of the solution to the elliptic variational inequality (\ref{eq:Varenq})~\cite{KinderlehrerStampacchia1980}, we have
\[u_a=u_b.\]

Finally, applying \eqref{eq:weakBeq} to the two pressure histories and
subtracting the resulting equations gives
\[
    \phi\int_D
    \left(
    \boldsymbol{1}_{\Omega_a(t)}
    -
    \boldsymbol{1}_{\Omega_b(t)}
    \right)w\,ds
    =0,
    \qquad w\in V_0(D),
\]
since $u_a(t)=u_b(t)$. Then, since the identity holds for all 
$w\in V_0(D)$, we have
\[
    \boldsymbol{1}_{\Omega_a(t)}
    =
    \boldsymbol{1}_{\Omega_b(t)}
    \qquad\text{a.e. in }D.
\]
That is to say, the filled region at time $t$ depends on the gate-pressure history
only through the cumulative gate pressures $\sigma(t)$, completing the
proof.
\end{proof}

\begin{remark}
The theorem is posed in terms of the filled-region
indicator functions and does not require the free boundary to be connected or smoothly
parameterised. Consequently, merging of
disjoint flow fronts, does not affect the result.
\end{remark}

The restriction to constant continuation pressures in Section~\ref{sec:Surrogate} and the controllability test of Section~\ref{sec:ForkResults} require a statement about the resin domain at the moment filling terminates, rather than at a fixed time. It follows from Theorem~\ref{thm:avgpressure} together with a comparison principle for the variational inequality \eqref{eq:Varenq}. Note that the proof of Theorem~\ref{thm:avgpressure} uses the initial resin domain only through $\boldsymbol{1}_{\Omega(0)}$, so the theorem applies verbatim from any initial time $t_0 \ge 0$ and initial domain $\Omega(t_0)$, with $\sigma(t) = \int_{t_0}^{t} a(\tau)\,d\tau$. Throughout, $\Omega(\sigma)$ denotes the resin domain produced by cumulative gate pressures $\sigma \in \mathbb{R}^{n_g}_{+}$ from a fixed initial domain, which by Theorem~\ref{thm:avgpressure} is well defined up to a set of measure zero.

\begin{lemma}[Comparison]
\label{lem:comparison}
Under the assumptions of Theorem~\ref{thm:avgpressure}, if $\sigma_1 \le \sigma_2$ componentwise, then $\Omega(\sigma_1) \subseteq \Omega(\sigma_2)$.
\end{lemma}
\begin{proof}
Two facts about the time-integrated pressure are needed. First, \eqref{eq:keyEq} gives $\{s \in D : u(s,t) > 0\} \subseteq \Omega(t)$. Conversely, if $s \in \Omega(t)\setminus\Omega(0)$, the flow front passed $s$ at some time $\tau_0 < t$; the front advances only where the pressure is not identically zero on the connected component of the resin domain containing it, and for times shortly after $\tau_0$ the point $s$ lies in the interior of that component, where the strong maximum principle gives $p(s,\tau) > 0$, so $u(s,t) > 0$. Hence
\begin{equation}
    \Omega(t) = \Omega(0) \cup \{ s \in D : u(s,t) > 0 \}
    \label{eq:domainfromu}
\end{equation}
up to a set of measure zero. Second, the variational inequality \eqref{eq:Varenq} has boundary data $\sigma_j$ on the gates and zero on the vent, and a source term $\phi(1-\boldsymbol{1}_{\Omega(0)})$ that does not depend on $\sigma$; the solution of an obstacle problem with a coercive bilinear form depends monotonically on its boundary data~\cite{KinderlehrerStampacchia1980}, so $\sigma_1 \le \sigma_2$ implies $u_1 \le u_2$ in $D$. Then $\{u_1 > 0\} \subseteq \{u_2 > 0\}$, and the claim follows from \eqref{eq:domainfromu}.
\end{proof}

\begin{corollary}[Sufficiency of constant policies]
\label{cor:constant}
Let $\mathcal{A}$ be a convex set of admissible gate pressures, let $\Omega(t_0)$ be the resin domain at some time $t_0 \ge 0$, and let $a(\cdot)$ be any admissible policy on $[t_0, T]$, $a(\tau) \in \mathcal{A}$, under which the resin first reaches the vent at time $T$. Then the constant policy equal to the time average $\bar{a} = \frac{1}{T - t_0}\int_{t_0}^{T} a(\tau)\,d\tau$ is admissible, the resin first reaches the vent under it at the same time $T$, and it produces the same terminal resin domain.
\end{corollary}
\begin{proof}
Admissibility follows from the convexity of $\mathcal{A}$, since $\bar{a}$ is an average of elements of $\mathcal{A}$. Let $\sigma(t) = \int_{t_0}^{t} a(\tau)\,d\tau$. Under the constant policy the cumulative pressures are $\sigma_c(t) = \frac{t - t_0}{T - t_0}\,\sigma(T)$, so $\sigma_c(T) = \sigma(T)$ and, by Theorem~\ref{thm:avgpressure}, the two policies produce the same domain at time $T$, which contains the vent. It remains to show that the vent is not reached earlier under the constant policy. Fix $t < T$ and set $\rho = \sigma_c(t)$, so that $\rho_j < \sigma_j(T)$ for every gate $j$ with $\sigma_j(T) > 0$, and $\rho_j = 0 = \sigma_j(\tau)$ for all $\tau$ otherwise. Because each component of $\sigma(\cdot)$ is continuous and non-decreasing, the time $t_j = \inf\{\tau : \sigma_j(\tau) \ge \rho_j\}$ satisfies $t_j < T$ for every $j$, so at $t^\star = \max_j t_j < T$ we have $\sigma(t^\star) \ge \rho$ componentwise. By Lemma~\ref{lem:comparison}, $\Omega(\rho) \subseteq \Omega(\sigma(t^\star))$, and the latter does not contain the vent because $t^\star < T$. Hence the constant policy does not reach the vent before $T$.
\end{proof}

\section{Training history for surrogates}
\label{app:hTraining}

Here, we show the training and validation loss over iterations for the two surrogate models $g$ and $h$, as described in Section~\ref{sec:NN}.

\begin{figure}[H]
\centering
\begin{minipage}[t]{0.48\textwidth}
\centering
\pgfplotstableread[col sep=comma]{Data/control_metrics.csv}\controlmetrics
\begin{tikzpicture}
\begin{axis}[
    width=\textwidth,
    height=0.55\textwidth,
    xlabel={Iteration},
    ylabel={Loss (MSE)},
    ymode=log,
    grid=major,
    unbounded coords=discard,
    filter discard warning=false,
    legend style={
        at={(0.5,1.02)},
        anchor=south,
        font=\scriptsize,
        draw=none,
        fill=none
    }
]
\addplot[teal, thick] table[x=iter,y expr={\thisrow{valLoss}>0 ? \thisrow{valLoss} : nan}]{\controlmetrics};
\addlegendentry{Validation Loss}
\addplot[blue, thick] table[x=iter,y expr={\thisrow{trainLoss}>0 ? \thisrow{trainLoss} : nan}]{\controlmetrics};
\addlegendentry{Training Loss}
\end{axis}
\end{tikzpicture}
\subcaption{Control surrogate $h$}
\label{fig:hTrainingLoss}
\end{minipage}
\hfill
\begin{minipage}[t]{0.48\textwidth}
\centering
\pgfplotsset{compat=1.18}
\pgfplotstableread[col sep=comma]{Data/training_metrics1.csv}\trainingone
\pgfplotstableread[col sep=comma]{Data/training_metrics2.csv}\trainingtwo
\pgfplotstablegetrowsof{\trainingone}
\pgfmathsetmacro{\trainingonesize}{\pgfplotsretval}
\pgfmathtruncatemacro{\switchepoch}{\trainingonesize}
\begin{tikzpicture}
\begin{axis}[
    width=\textwidth,
    height=0.55\textwidth,
    xlabel={Iteration},
    ylabel={Loss (MSE)},
    ymode=log,
    ymin=5e-4,
    ymax=3e-2,
    grid=major,
    xmin = 0,
    xmax = 625,
    legend style={
        at={(0.5,1.02)},
        anchor=south,
        font=\scriptsize,
        draw=none,
        fill=none
    }
]
\addplot[teal, thick] table[x expr = \thisrow{iter}/100,y=valLoss]{\trainingone};
\addlegendentry{Validation Loss}
\addplot[blue, thick] table[x expr = \thisrow{iter}/100,y=trainLoss]{\trainingone};
\addlegendentry{Training Loss}
\addplot [black, dashed, thick] coordinates {(\switchepoch,1e-4) (\switchepoch,1e-1)};
\addlegendentry{Data split}
\addplot[teal, thick] table[x expr=\switchepoch + \thisrow{iter}/100+1,y=valLoss]{\trainingtwo};
\addplot[blue, thick] table[x expr=\switchepoch + \thisrow{iter}/100+1,y=trainLoss]{\trainingtwo};
\end{axis}
\end{tikzpicture}
\subcaption{Surrogate $g$}
\label{fig:NNtrainingloss}
\end{minipage}
\caption{Training and validation loss for control surrogate $h$ (a) and surrogate $g$ (b) over training iterations. For $g$, training proceeded in two stages; the vertical dashed line indicates the point at which the ReLU activation layer was added to the output layer. The validation loss decreases significantly after the first stage, and continues to improve after the second stage, consistent with the two-stage training procedure.}
\label{fig:trainingLosses}
\end{figure}
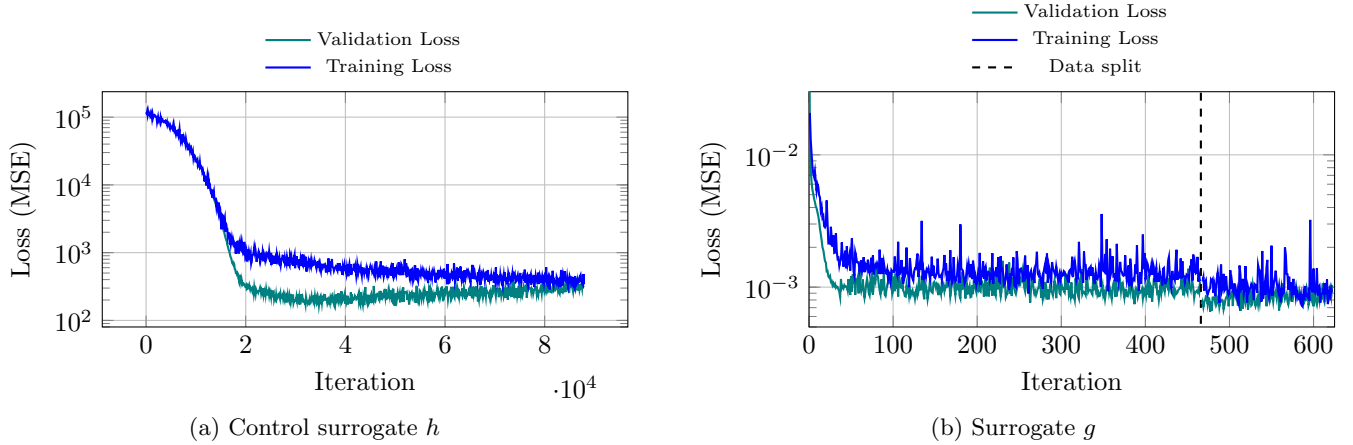
\end{document}